\documentclass[12pt]{article}

\usepackage{amsmath, amsthm}
\usepackage{amscd}
\usepackage{amssymb}
\usepackage{array}
\usepackage{amsfonts}
\usepackage{mathtools}
\usepackage{color}
\usepackage{graphicx}
\usepackage{xfrac}
\usepackage{authblk}
\usepackage{ytableau}
\usepackage{array}

\usepackage{enumitem}

\newtheorem{thm}{Theorem}[section]
\newtheorem{theorem}[thm]{Theorem}

\newtheorem{lemma}[thm]{Lemma}
\newtheorem{conjecture}[thm]{Conjecture}
\newtheorem{proposition}[thm]{Proposition}
\theoremstyle{definition}
\newtheorem{definition}[thm]{Definition}
\newtheorem{example}[thm]{Example}

\newtheorem{remark}[thm]{Remark}

\newtheorem{observation}[thm]{Observation}

\newcommand{\Z} {\mathbb{Z}}
\newcommand{\C} {\mathbb{C}}

\newcommand{\R}{\mathbb{R}}

\newcommand{\cA}{\mathcal{A}}

\newcommand{\cO}{\mathcal{O}}

\newcommand{\rGL}{{\mathrm{GL}}}

\newcommand{\rSL}{\mathrm{SL}}

\newcommand{\Hom}{\mathrm{Hom}}

\newcommand{\rM}{\mathrm{M}}

\newcommand{\Ber}{\mathrm{Ber}}

\newcommand{\beq}{\begin{equation}}
\newcommand{\eeq}{\end{equation}}

\title{{\bf On Fundamental Theorems of Super Invariant Theory}}
\begin{document}

\centerline{{\Large \bf On Fundamental Theorems of Invariant Theory}}

\bigskip
\centerline{{\Large \bf for the Special Linear Supergroup}}

\vskip 1cm
\bigskip

\centerline{ J. Razzaq }

\smallskip
\centerline{\it Mathematical Institute of Charles University.}
 \centerline{\it Sokolovsk\'a 83, Prague, Czech Republic.}

\centerline{{\footnotesize e-mail: junaid.razzaq@matfyz.cuni.cz}}

\bigskip

\centerline{ R. Fioresi}
\smallskip
\centerline{\it Fabit, Universit\`a di
Bologna.}

{\centerline{\it Via San Donato, 15. 40126 Bologna, Italy.}}

\centerline{\it and INFN,  Sezione di
    Bologna, Italy}

\centerline{{\footnotesize e-mail: rita.fioresi@unibo.it}}

\bigskip
\centerline{ M. A. Lled\'o }

\smallskip

 \centerline{\it  Departament de F\'{\i}sica Te\`{o}rica,
Universitat de Val\`{e}ncia and }

\centerline{\it IFIC (CSIC-UVEG).}
 \centerline{\it C/Dr.
Moliner, 50, E-46100 Burjassot (Val\`{e}ncia), Spain.}
 \centerline{{\footnotesize e-mail: maria.lledo@ific.uv.es}}

\vskip 1cm

\begin{abstract}
The purpose of this paper is to prove the First and Second Fundamental Theorems of invariant theory for
the complex special linear supergroup and discuss the superalgebra of invariants,
via the super Pl\"ucker relations.
\end{abstract}

\section{Introduction}

Since its origins, in the
XIXth century \cite{cayley},
classical invariant theory rapidly provided important results,
whose applications involve both, mathematics and physics
(see \cite{borel} for an historical review of the theory).

One of the classical invariant theory main questions,
which led to the spectacular result by Hilbert \cite{hilbert}, is the search for invariants,
with respect to a given group action, in the ring of polynomial functions
(\cite{schur, weyl}).

We can look, for example, at the polynomial functions on the set of $r \times p$
$(r\leq p)$ complex matrices $\mathrm{M}_{r \times p}$, that
are invariant under the action of the complex
special linear group $\mathrm{SL}_r(\C)$:
$$\begin{CD}
\mathbb{C}[\mathrm{M}_{r \times p}] \times \mathrm{SL}_r(\C) @>>>
\mathbb{C}[\mathrm{M}_{r \times p}]\\ (f,g) @>>> f.g, \end{CD}$$ where
$(f.g)(M):=f(gM)
$
and $\mathbb{C}[\mathrm{M}_{r \times p}]$ denotes the algebra of polynomials
functions on the entries of $\mathrm{M}_{r \times p}$.

Let $X_{i_1\dots i_r}$ denote the  $r\times r$ minor formed with the columns $(i_1,\dots , i_r)$ of a matrix in $\mathrm{M}_{r \times p}$.
The First and Second fundamental Theorems of invariant
theory for $\mathrm{SL}_r(\C)$
state the following (see \cite{gw, fulton-h, Fulton} { and Refs. therein}
for a full account).

\begin{theorem} \label{fsft}
Let the notation be as above.
\begin{enumerate}
\item {\bf First Fundamental Theorem (FFT) of invariant theory.}
The ring of invariants $\mathbb{C}[\mathrm{M}_{r \times p}]^{\mathrm{SL}_r(\C)}$
is generated by the minors $X_{i_1\dots i_r}$ of the $r \times r$ submatrices in
$\mathrm{M}_{r \times p}$.
\item {\bf Second Fundamental Theorem (SFT) of invariant theory.}
We have a presentation of the ring of invariants via the ideal $\mathrm{I}$ of the
Pl\"ucker relations:
\begin{align}
&\mathbb{C}[\mathrm{M}_{r \times p}]^{\mathrm{SL}_r(\C)} \cong
\mathbb{C}[X_{i_{1} \dots i_{r}}]\,\big/\,\mathrm{I},\qquad \hbox{            with}\nonumber\\[0.3cm]
&\mathrm{I}:=\left(\sum_{k=1}^{r+1}(-1)^{k} X_{i_{1}...i_{r-1}j_{k}}X_{j_{1} \dots \tilde{\jmath}_{k} \dots j_{r+1}}\right) 
\label{ideal-plucker}
\end{align}
where
\begin{align*}&1 \leq i_1 <\cdots< i_{r-1} \leq  p\\
&1 \leq j_{1}<\cdots<\tilde{\jmath}_{k}< \dots <j_{r+1}\leq p
\end{align*}and
$\tilde{\jmath}_{k}$ means that the index is removed.
\end{enumerate}
\end{theorem}

In this work we want to provide a generalization of this result to supergeometry.\hyphenation{su-per-geo-me-try}
The generalization of some questions of invariant theory to the super setting appeared
early in the literature: for example, the Young super tableaux
 and techniques regarding their manipulation, appeared first in
\cite{jarvis} and at the same time in \cite{br}, to address questions of
representation theory of Lie superalgebras and the double commutant theorem, a fundamental result, originally proven, in the ordinary setting,
by Schur \cite{schur}.

In particular,\hyphenation{re-pre-sen-ta-tion}
in \cite{br}, the authors prove also a version of
the straightening algorithm for Young super tableaux, of uttermost
importance in representation theory and directly linked to the Pl\"ucker relations \cite{Fulton}.
Later on, in the papers \cite{rota, brini},
appeared such algorithm in the super setting.
Their idea, involving the so called {\it virtual variables method}, provides an elegant approach
which encompasses both, the ordinary and the super setting at once.
More recently in \cite{zhang, deligne} appeared generalizations to other supergroups
as the orthosymplectic one and in \cite{coul} a categorical approach to super invariant
theory allowed to go beyond
the characteristic zero setting (see also \cite{coul-ann}).

\hyphenation{dif-fe-rent}

In the present work, however, we are interested in a related, and yet different
topic.
It has its significance rooted into the geometric meaning of
these combinatorial questions. In the ordinary setting Theorem \ref{fsft} provides with
a presentation of the ring of invariants via the Pl\"ucker relations, which
give an embedding of the Grassmannian manifold into a suitable projective space.
In supergeometry, super Grassmannians do not admit, in general,
projective embeddings (see \cite{Manin}, Chapter 4).
However, in \cite{Voronov}, the authors define a new graded
version of super projective space, with negative grading, which allows for such embedding.
Indeed, the Berezinians of sub-supermatrices replace truly the notion of  minors,
which give the projective coordinates
for the classical Pl\"ucker embedding. Then, they allow to proceed
and give the correct supergeometric counterpart of the Pl\"ucker embedding and Pl\"ucker relations \cite{Voronov}. Furthermore, in \cite{Shemyakova}, these super Pl\"ucker relations has been studied from the perspective of constructing super cluster algebras. However, due to the very nature of Berezinians, which are defined only when  the supermatrix is invertible, one is forced to restrict the analysis to an open set. Once this geometric constraint is set in place, one can state and prove the super version of the {First Fundamental Theorem} and progress towards the full proof of the more elusive {Second Fundamental Theorem}.

\medskip
The paper is organized as follows.

\hyphenation{Be-re-zi-nians}
In Section \ref{prelim-sec}, we introduce the notation and a few
facts concerning the Berezinians. We have included in Appendix \ref{super-app}
some very basic notions of supergeometry. We also state the super Cramer rule,
first proved in  \cite{Bergvelt} { with different methods},
which is instrumental for the proofs of the results in the next sections. Due to its importance to the present work, and in order to be self contained, we have included { our}
new proof of super Cramer's rule in Appendix \ref{cramer-app}, which is very much in line with the reasoning of the rest of the paper.

In Section \ref{fft-sec}, we state and prove the {\sl First Fundamental
Theorem of super invariant theory} for the special linear supergroup, Theorem \ref{SFFT}.

In Section \ref{super-jac}, we introduce and prove a super version
of the Jacobi identiy (a classical determinant identity, for details see Appendix \ref{jacobi-app}), Theorem \ref{superjacobi}, one of our main results, which is instrumental for the following sections.

In {Sections \ref{pl-11}, \ref{sec-sft}}
we derive the {\sl super Pl\"ucker relations}  from the super Jacobi identity in a special,
yet illustrative, case (i.e. of $\mathrm{SL}(1|1)$). We also prove {\sl Second Fundamental Theorem of super invariant theory} for $\mathrm{SL}(1|1)$.

In Section \ref{sft-sec}
we derive the {\sl super Pl\"ucker relations} for the case of $\mathrm{SL}(r|s)$ and conjectured {\sl Second Fundamental Theorem of super invariant theory} for $\mathrm{SL}(r|s)$. It is important to note that the super Pl\"ucker relations we derived are same as the one given in \cite{Voronov}, however, we used a different construction based on super Jacobi identity.

\section{Preliminaries} \label{prelim-sec}
In this section we establish our notation on supermatrices. The reader can go to Appendix \ref{super-app} or to references \cite{ccf, fl, Voronov} for
terminology.

Let $A$ be a fixed commutative superalgebra. All entries
of supermatrices are intended to be in $A$.
An \textit{even supermatrix}
\begin{equation} \label{superm}
B= \begin{pmatrix}
    B_{1} &\vline& B_{2}\\
    \hline
    B_{3} &\vline& B_{4}
\end{pmatrix}
\end{equation}
of size $p|q \times r|s$  is a block matrix whose first $r$ columns are even vectors and last $s$ columns are odd vectors. Consequently, $B_{1}$ and $B_{4}$ contain even entries, while $B_{2}$ and $B_{3}$ contain odd entries.

\medskip

We also need to introduce the concept of {\it fake supermatrices}\footnote{These are called `wrong supermatrices' in \cite{Voronov}.}.
\begin{definition}
 A \textit{fake supermatrix of type-I}  is a matrix obtained by inserting an odd vector at an even vector's position of an even supermatrix.
 A \textit{fake supermatrix of type-II}  is a matrix obtained by inserting an even vector at an odd vector's position of an even supermatrix.
\end{definition}

As usual, for any even supermatrix or fake supermatrix, as in (\ref{superm}), we define its {\it parity reversed matrix},
denoted by $B^{\Pi}$ by swapping $B_{1}$ with $B_{4}$ and $B_{2}$ with $B_{3}$, i.e.
$$
\begin{pmatrix}    B_{1} &\vline& B_{2}\\ \hline B_{3} &\vline& B_{4} \end{pmatrix}^{\Pi}=\begin{pmatrix}
    B_{4} &\vline& B_{3}\\
    \hline
    B_{2} &\vline& B_{1}
\end{pmatrix}.
$$
Notice that if $B$ is a fake supermatrix of type I, then $B^\Pi$ is a fake supermatrix of type II and viceversa.
\begin{definition}
Let $B$ be an even supermatrix or a fake supermatrix of type-I:
\begin{equation} \label{superm1}
B= \begin{pmatrix}
    B_{1} &\vline& B_{2}\\
    \hline
    B_{3} &\vline& B_{4}
\end{pmatrix}.
\end{equation}
Assume that $B_4$ is invertible.
We define the \textit{Berezinian} of $B$ as:
$$
\Ber B:=\det B_{4}^{-1}\det(B_{1}-B_{2}B_{4}^{-1}B_{3})
$$
\end{definition}
Notice that, although this definition is the same as the first line of (\ref{bereq}), in Appendix \ref{super-app},
here we use it in a more general context, when $B$ can be a fake supermatrix of type-I.

\begin{definition}
Let $B$ be an even supermatrix or a fake supermatrix of type-II:
\begin{equation} \label{superm2}
B= \begin{pmatrix}
    B_{1} &\vline& B_{2}\\
    \hline
    B_{3} &\vline& B_{4}
\end{pmatrix}
\end{equation}
Assume that $B_1$ is invertible.
We define the $\textit{inverted berezinian}$ of $B$:
$$
\Ber^{*}B:= \Ber B^{\Pi}.
$$
It is clear that for an invertible supermatrix $B$:
\begin{equation}
\Ber^{*}B=\Ber B^{-1}. \label{ob}
\end{equation}
\end{definition}

\begin{remark}
For  an even supermatrix $B$,
we denote with $i$ the even indices and with $\hat{\imath}$ the odd ones, as we exemplify below:
$$
B=\begin{pmatrix}
    b_{11}& ... &b_{1p} &\vline& b_{1\hat{1}}& ... &b_{1\hat{q}}\\
    .& ... &. &\vline& .& ... &.\\
    .& ... &. &\vline& .& ... &.\\
    b_{r1}& ... &b_{rp} &\vline& b_{r\hat{1}}& ... &b_{r\hat{q}}\\
    \hline
    b_{\hat{1}1}& ... &b_{\hat{1}p} &\vline& b_{\hat{1}\hat{1}}& ... &b_{\hat{1}\hat{q}}\\
    .& ... &. &\vline& .& ... &.\\
    .& ... &. &\vline& .& ... &.\\
    b_{\hat{s}1}& ... &b_{\hat{s}r} &\vline& b_{\hat{s}\hat{1}}& ... &b_{\hat{s}\hat{q}}\\
\end{pmatrix}.$$
\end{remark}

\bigskip
\bigskip

We state now the super version of Cramer's rule. Recall that, in ordinary linear algebra, for an invertible matrix $T$ of size $n \times n$ and a column vector $\mathbf{b}$ of size $n$, the solution to the equation:
$$T\mathbf{x}=\mathbf{b}$$
is a column vector $\mathbf{x}$ whose entries are given by:
$$x_{i}=\dfrac{\det T_{i}(\mathbf{b})}{\det T}$$
where $T_{i}(\mathbf{b})$ denotes the matrix obtained by replacing the $i$-th column of $T$ with $\mathbf{b}$.

\medskip

In super linear algebra, we have the following theorem:

\begin{theorem}{\sl Super Cramer's rule}.\label{supercramer}
Let $M\in \mathrm{GL}(r|s)$ be an invertible even supermatrix and let $\mathbf{b}\in \mathbf{k}^{r|s}$ be an even vector of size $r|s$. Then, the solution to the equation
$$M\mathbf{x}=\mathbf{b}$$
is given by:
\begin{align*}
x_{i}=\dfrac{\Ber M_{i}(\mathbf{b})}{\Ber M}\qquad \hbox{for}\quad i=1,...,r
\end{align*}
\begin{align*}
x_{\hat{\jmath}}=\dfrac{\Ber^{*} M_{\hat{\jmath}}(\mathbf{b})}{\Ber^{*}M}\qquad \hbox{for}\quad\hat{\jmath}=\hat{1}\cdots \hat{s}
\end{align*}
where $M_{i}(\mathbf{b})$ is the supermatrix obtained by replacing $i$-th even column of $M$ with $\mathbf{b}$ and $M_{\hat{\jmath}}(\mathbf{b})$ is the fake supermatrix of type-II obtained by replacing $\hat{\jmath}$-th odd column of $M$ with $\mathbf{b}$.
\end{theorem}

This is a known result, see \cite{Bergvelt}. However, for completeness, we have included in Appendix \ref{cramer-app} a proof of it, obtained with different, elementary methods.
\hyphenation{me-thods}

As stated in Remark \ref{cramerfake}, the same expressions hold for the equation $M\mathbf{x}=\mathbf{b}$ if $\mathbf{b}$ is an odd vector of size $r|s$.

\begin{example}
Let $g=\begin{pmatrix}
    a &\vline& \alpha\\
    \hline
    \beta &\vline& b
\end{pmatrix} \in \mathrm{GL}(1|1)$. Suppose we need to find $\begin{pmatrix}
    x\\
    \hline
    \eta
\end{pmatrix}$ such that:
$$
g\begin{pmatrix}
    x\\
    \hline
    \eta
\end{pmatrix}= \begin{pmatrix}
    c\\
    \hline
    \gamma
\end{pmatrix}.$$
Then, using super Cramer's rule, one gets:
$$x= \dfrac{\Ber \begin{pmatrix}
    c &\vline& \alpha\\
    \hline
    \gamma &\vline& b
\end{pmatrix}}{\Ber \begin{pmatrix}
    a &\vline& \alpha\\
    \hline
    \beta &\vline& b
\end{pmatrix}}= \dfrac{c-\alpha b^{-1}\gamma}{a-\alpha b^{-1} \beta}$$
and
$$\eta=\dfrac{\Ber^{*} \begin{pmatrix}
    a &\vline& c\\
    \hline
    \beta &\vline& \gamma
\end{pmatrix}}{\Ber^{*} \begin{pmatrix}
    a &\vline& \alpha\\
    \hline
    \beta &\vline& b
\end{pmatrix}}= \dfrac{\gamma-\beta a^{-1}c}{b-\beta a^{-1} \alpha}.$$
\end{example}

\section{First Fundamental Theorem of invariant \\theory for $\mathrm{SL(r|s)}$}\hyphenation{theo-ry}
\label{fft-sec}
\hyphenation{theo-ry}
In ordinary geometry, we have a natural right action of the complex special linear group $\rSL_r(\C)$
on the algebra of functions $\C[\rM_{r \times p}]$ of the $r \times p$ complex matrices $\rM_{r \times p}$:

\begin{equation}\label{eq-action}\begin{CD}
\C[\rM_{r \times p}]\times\rSL_r(\C)@>>>\C[\rM_{r \times p}]\\
(f,g)@>>>f.g,\end{CD}\end{equation}
where $(f.g)(M):=f(gM)$, for all $M \in \rM_{r\times p}(\C)$.

The First Fundamental Theorem of ordinary invariant theory realizes the ring of
invariants with respect to this action,
as the subring of $\C[\rM_{r \times p}]$ generated by the determinants
of the $r \times r$ minors in $\rM_{r\times p}(\C)$.

We now extend the action (\ref{eq-action}) to the super setting and compute
the superalgebra of invariants. Let $\mathrm{M}_{r|s \times p|q}$ be the super vector
space of supermatrices and let $\cO$ be the superalgebra of  `functions' on
$\mathrm{M}_{r|s \times p|q}$ (see Observation \ref{fopt-yo}).
For notational purposes, let us organize the generators  of $\cO$  (or `coordinates') in a supermatrix, as follows:

\begin{equation}A=\label{blocks}\begin{pmatrix}
x_{11} & ... & x_{1p} &\vline& \alpha_{1\hat{1}} &... & \alpha_{1\hat{q}}\\
. & ... & . &\vline& . &... & .\\
. & ... & . &\vline& . &... & .\\
x_{r1} & ... & x_{rp} &\vline& \alpha_{r\hat{1}} &... & \alpha_{r\hat{q}}\\
\hline
\beta_{\hat{1}1} & ... & \beta_{\hat{1}p} &\vline & y_{\hat{1}\hat{1}} & ...& y_{\hat{1}\hat{q}}\\
. & ... & . &\vline& . &... & .\\
. & ... & . &\vline& . &... & .\\
\beta_{\hat{s}1} & ... & \beta_{\hat{s}p} &\vline & y_{\hat{s}\hat{1}} & ...& y_{\hat{s}\hat{q}}
\end{pmatrix},
\end{equation}
where the hat denotes the odd indices.
Then,
$$\mathcal{O}= k[x_{ij},y_{\hat{k}\hat{l}},\alpha_{i\hat{k}},\beta_{\hat{l}j}].$$ Let $\mathcal{A}$ be a commutative superalgebra. For every supermatrix in the set of $\mathcal{A}$-points of $\mathrm{M}_{r|s \times p|q}$ (see Appendix \ref{super-app}), $M \in \mathrm{M}_{r|s \times p|q}(\mathcal{A})$ and
$f \in \cO$ one can compute an element of $\mathcal{A}$. This is denoted, in the terminology of the $A$-points, as $f(M)$ instead of the, perhaps more appropriate,
$M(f)$.

Let $\rSL(r|s)$ be the special linear
supergroup (see, for example, \cite{ccf}, Chapter 1). Its $\mathcal{A}$-points are the matrices in $\rGL(r|s)(\mathcal{A})$
with Berezinian equal to 1.
In terms of $\mathcal{A}$-points, one can write a left action of the special linear supergroup on the supermatrices:
$$\begin{CD}\rSL(r|s)(\mathcal{A})\times\mathrm{M}_{r|s \times p|q}(\mathcal{A})@>>>\mathrm{M}_{r|s \times p|q}(\mathcal{A})\\
(g, M)@>>>gM.\end{CD}
$$
One can also define a right action of $\rSL(r|s)$ on $\cO$, exactly as in the ordinary case (\ref{eq-action}). For $g \in \rSL(r|s)(\mathcal{A})$ and $f \in \cO$ one has
$$
(f.g)(M):=f(gM), \quad M \in \rM_{r|s\times p|q}(\mathcal{A}).
$$

\medskip

We want to be able to define the Berezinians of all $r|s \times r|s$ minors appearing in the matrix (\ref{blocks}). We then have to localize  $\cO$ at the determinants
\begin{equation*}
 D_{i_{1},..,i_{r}}:=\det \begin{pmatrix}
     x_{1i_{1}} &...& x_{1i_{r}}\\
     . &...&.\\
     . &...&.\\
     x_{ri_{1}} &...& x_{ri_{r}}\\
 \end{pmatrix},\qquad
 D_{\hat{\jmath}_{1}\cdots \hat{\jmath}_{s}}:=
 \det \begin{pmatrix}
     y_{\hat{1}\hat{\jmath}_{1}} &\cdots& y_{\hat{1}\hat{\jmath}_{s}}\\
     \cdot &\cdots&\cdot\\
     \cdot &\cdots&\cdot\\
     y_{\hat{s}\hat{\jmath}_{1}} &\cdots& y_{\hat{s}\hat{\jmath}_{s}}\\
 \end{pmatrix},
\end{equation*}
in the diagonal blocks, in the expression (\ref{blocks}).
We can now look at the action of the supergroup $\rSL(r|s)$ on the localization
\begin{equation}\label{otilde}
\Tilde{\mathcal{O}}:=\mathcal{O}[D_{i_{1}\cdots i_{r}}^{-1}, D_{\hat{\jmath}_{1}\cdots \hat{\jmath_{s}}}^{-1}].
\end{equation}
This corresponds to take the subset $\Tilde \rM_{r|s\times p|q} (\cA)\subset\rM_{r|s\times p|q}(\mathcal{A})$ where all the corresponding minors are invertible.

In this way we can formulate the super version of the First Fundamental Theorem of
invariant theory. As we shall see in the next section, it is related with
the Grassmannian and its super Pl\"ucker embedding, and we adopted a very similar notation for generators as in \cite{Voronov}.

\begin{theorem}\label{SFFT}
Let $A_{i_{1}\cdots i_{r}|\hat{\jmath}_{1}\cdots \hat{\jmath}_{s}}$ be the supermatrix formed with the columns $(i_{1},\dots, i_{r}|\hat{\jmath}_{1},\dots, \hat{\jmath}_{s})$ of the supermatrix $A$ in (\ref{blocks}). Then, the ring of invariants $\Tilde{\mathcal{O}}^{\mathrm{SL}(r|s)}$
is generated by the { super minors}
\begin{align}
&X_{i_{1}\cdots i_{r}|\hat{\jmath}_1\cdots \hat{\jmath}_{s}}:=\Ber A_{i_{1}\cdots i_{r}|\hat{\jmath}_{1}\cdots \hat{\jmath}_{s}} \nonumber\\ &X^{*}_{i_{1}\cdots i_{r}|\hat{\jmath}_{1}\cdots \hat{\jmath}_{s}}:=\Ber^{*}A_{i_{1}\cdots i_{r}|\hat{\jmath}_{1}\cdots \hat{\jmath}_{s}}\label{superminor1}
\end{align}
and the {`fake' super minors},
\begin{align}\label{superminor2}
X_{i_{1}\cdots i_{k-1}\hat{\jmath}_{k} i_{k+1}\cdots i_{r}|\hat{\jmath}_{1}\cdots \hat{\jmath}_{s}}&:=
\Ber A_{i_{1}\cdots i_{k-1}\hat{\jmath}_{k} i_{k+1}\cdots i_{r}|\hat{\jmath}_{1}\cdots \hat{\jmath}_{s}}
\\\label{superminor3}
X^{*}_{i_{1}\cdots i_{r}|\hat{\jmath}_{1}\cdots \hat{\jmath}_{l-1}i_{l}\hat{\jmath}_{l+1}\cdots\hat{\jmath}_{s}}&:= \Ber^{*}A_{i_{1}\cdots i_{r}|\hat{\jmath}_{1}\cdots \hat{\jmath}_{l-1}i_{l}\hat{\jmath}_{l+1}\cdots \hat{\jmath}_{s}}
\end{align}
where $ 1\leq i_{1} \leq \cdots\leq i_{r} \leq p $ and $1\leq \hat{\jmath}_{1} \leq \cdots\leq \hat{\jmath}_{s} \leq q$.
\end{theorem}

\begin{proof}
Firstly, it is clear, from the multiplicative property of Berezinian, that super-minors and fake super-minors are invariants.

To show  that they generate the ring $\Tilde{\mathcal{O}}^{\rSL(r|s)}$,  we decompose the matrix of coordinate functions (\ref{blocks}) as $A=\Tilde{A}_{1\cdots r|\hat{1}\cdots \hat{s}}B$, where

$$
\Tilde{A}_{1\cdots r|\hat{1}\cdots \hat{s}}= \begin{pmatrix}
\dfrac{x_{11}}{X_{1\cdots r|\hat{1}\cdots \hat{s}}} &\cdots  & x_{1r} &\vline& \alpha_{1\hat{1}} &\cdots  & \alpha_{1\hat{s}}\\
\cdot & \cdots  &\cdot &\vline& \cdot &\cdots  & \cdot\\
\cdot & \cdots  & \cdot &\vline& \cdot &\cdots  & \cdot\\
\dfrac{x_{r1}}{X_{1\cdots r|\hat{1}\cdots \hat{s}}} &\cdots  & x_{rr} &\vline& \alpha_{r\hat{1}} &\cdots  & \alpha_{r\hat{s}}\\
&&&\vline&&&\\
\hline
&&&\vline&&&\\
\dfrac{\beta_{\hat{1}1}}{X_{1\cdots r|\hat{1}\cdots \hat{s}}} &\cdots & \beta_{\hat{1}r} &\vline & y_{\hat{1}\hat{1}} &\cdots& y_{\hat{1}\hat{s}}\\
\cdot & \cdots & \cdot &\vline& \cdot &\cdots & \cdot\\
\cdot & \cdots& \cdot &\vline& \cdot &\cdots & \cdot\\
\dfrac{\beta_{\hat{s}1}}{X_{1\cdots r|\hat{1}\cdots \hat{s}}} &\cdots & \beta_{\hat{s}r} &\vline & y_{\hat{s}\hat{1}} & \cdots& y_{\hat{s}\hat{s}}
\end{pmatrix}.
 $$
Therefore, we have,
$$
B=\Tilde{A}_{1\cdots r|\hat{1}\cdots \hat{s}}^{-1}A.
$$
{Notice that, we are working over $\mathcal{O}[D_{i_{1}\cdots i_{r}}^{-1}, D_{\hat{\jmath}_{1}\cdots \hat{\jmath_{s}}}^{-1}]$, hence, $\Tilde{A}_{1\cdots r|\hat{1}\cdots \hat{s}}^{-1}$ is well-defined.}
Using super Cramer's rule we can find the entries of $B=\begin{pmatrix}
        U &\vline& V\\
        \hline
        W &\vline& Z
    \end{pmatrix}$:

 \begin{align}
     &U_{ij}=
     \begin{cases}
   X_{j\cdots r|\hat{1}\cdots \hat{s}} &\hbox{for } i=1, 1\leq j \leq p,\\
     X^{*}_{1\cdots r|\hat{1}\cdots \hat{s}}X_{1\cdots (i-1)j(i+1)\cdots r|\hat{1}\cdots \hat{s}} &\hbox{for } 1< i\leq r,\, 1\leq j \leq p
     \end{cases}
 \\[0.3cm]
&W_{\hat{k}j}=
X_{1\cdots r|\hat{1}\cdots \hat{s}}X^{*}_{1\cdots r|\hat{1}\cdots (\hat{k}-\hat{1}) j(\hat{k}+\hat{1})\cdots \hat{s}} \qquad  \hbox{for } 1\leq k\leq s,\, 1\leq j \leq p.
\\[0.3cm]
  &  V_{i\hat{l}}=\begin{cases}
     X_{\hat{l}\cdots r|\hat{1}\cdots \hat{s}} &\hbox{for }i=1,\,1\leq l \leq q,\\
     X^{*}_{1\cdots r|\hat{1}\cdots \hat{s}}X_{1\cdots (i-1)\hat{l}(i+1)\cdots r|\hat{1}\cdots \hat{s}} &\hbox{for } 1< i\leq r,\,1\leq l \leq q.
     \end{cases}
\\[0.3cm]
&Z_{\hat{k}\hat{l}}=X_{1\cdots r|\hat{1}\cdots \hat{s}}X^{*}_{1\cdots r|\hat{1}\cdots( \hat{k}-\hat{1})\hat{l}(\hat{k}+\hat{1})\cdots \hat{s}} \qquad \hbox{for } 1\leq k\leq s,\, 1\leq l \leq q.
\end{align}
Let  $\cA$ be any superalgebra and consider an arbitrary element $T\in \Tilde M_{r|s\times p|q}(\cA)$. We can write $T$ as a product,
\begin{eqnarray}\label{R}
T=\Tilde T_{1\cdots  r|\hat{1}\cdots  \hat{s}}R
\end{eqnarray}
where,
$\Tilde T_{1\cdots  r|\hat{1}\cdots \hat{s}}$ is an element of $\rSL(r|s)(\cA)$ and $R\in \Tilde M_{r|s\times p|q}(\cA)$ is computed in terms of superminors and  fake superminors. Explicitly, each entry of $R$ can be computed by evaluating on $T$ the corresponding superminor or fake superminor as in $B$ above.

{Let $f\in \Tilde\cO^{\rSL(r|s)}$ be an invariant. Then, by definition, we have:
\begin{eqnarray*}
f.g(R)=f(gR)=f(R), \qquad \text{for all }R\in \Tilde M_{r|s\times p|q}(\cA) \text{ and }g \in \rSL(r|s)(\cA).
\end{eqnarray*}
But then, for any $T \in \Tilde M_{r|s\times p|q}(\cA)$,
$$f(T)=f(\Tilde T_{1\cdots r|\hat{1}\cdots \hat{s}}R)=f(R),$$
where $R$ is as in (\ref{R}). Therefore, $f$ is a function of superminors and fake superminors, as we wanted to prove.}
\end{proof}

\section{The super Jacobi identity}\label{super-jac}
In this section we want to obtain a super Jacobi identity, a generalization of  the  Jacobi identity, Theorem \ref{J2} in the Appendix \ref{jacobi-app},  to the super setting. We first introduce some notation.
\hyphenation{ge-ne-ra-li-za-tion}
\medskip

{\bf Notation.} Let $A$ be an invertible supermatrix of size $p|q \times p|q$. Let $u=(p-r+1,\dots  ,p|{\hat{q}-\hat{s}+\hat{1}},\dots  ,\hat{q})$ and $v=(1,\dots ,r|\hat{1},\dots  ,\hat{s})$ be two sets of indices, $r<p$, $\hat{s}<\hat{q}$.
Let $A^{u}_{v}$ denote the matrix obtained from $A$ by deleting the rows and columns
contained in $u$ and $v$, respectively, $i.e.$ by deleting the last $r$ even rows, the last $\hat{s}$ odd rows, first $r$ even columns and first $\hat{s}$ odd columns. Also, set $\Tilde{u}=(1,\dots  ,p-r|\hat{1},\dots  ,{\hat{q}-\hat{s}})$ and $\Tilde{v}=(r+1,\dots  ,p|\hat{s}+\hat{1},\dots  ,\hat{q})$ as the complements of $u$ and $v$, respectively.

\begin{theorem}{\sl The super Jacobi Identity}.\label{superjacobi}
Let the notation be as above.
If $A^{u}_{v}$ and $(A^{-1})^{\Tilde{v}}_{\Tilde{u}}$ are invertible, then
\begin{eqnarray} \label{SJacobi}
    \Ber A \,\Ber(A^{-1})^{\Tilde{v}}_{\Tilde{u}}= (-1)^{t} \Ber A^{u}_{v}
\end{eqnarray}
where $t=r(p+1)+s(q+1)$.
\end{theorem}
We call (\ref{SJacobi}) the \textit{super Jacobi identity}.

\begin{proof}
 A similar strategy to what we used in the proof of Theorem \ref{J2} (see Appendix \ref{jacobi-app}) will work here. Let $$A=\begin{pmatrix}
     [a_{ij}] &\vline& [a_{i\hat{k}}]\\
     \hline
     [a_{\hat{l}j}] &\vline& [a_{\hat{l}\hat{k}}]
 \end{pmatrix}\qquad\hbox{and}\qquad A^{-1}=\begin{pmatrix}
     [b_{ij}] &\vline& [b_{i\hat{k}}]\\
     \hline
     [b_{\hat{l}j}] &\vline& [b_{\hat{l}\hat{k}}]
 \end{pmatrix}.$$ According to our notation,
 \begin{equation*}
     (A^{-1})_{\Tilde{u}}^{\Tilde{v}}= \begin{pmatrix}
     b_{1(p-r+1)}&\cdots&b_{1p} &\vline& b_{1(\hat{q}-\hat{s}+\hat{1})}&\cdots&b_{1\hat{q}}\\
     \cdot&\cdots&\cdot &\vline& \cdot&\cdots&\cdot\\
      \cdot&\cdots&\cdot&\vline& \cdot&\cdots&\cdot\\
      b_{r (p-r+1)}&\cdots&b_{rp} &\vline& b_{r(\hat{q}-\hat{s}+\hat{1})}&\cdots&b_{r\hat{q}}\\
      \hline
      b_{\hat{1} (p-r+1)}&\cdots&b_{\hat{1}p} &\vline& b_{\hat{1}(\hat{q}-\hat{s}+\hat{1})}&\cdots&b_{\hat{1}\hat{q}}\\
      \cdot&\cdots&\cdot &\vline& \cdot&\cdots&\cdot\\
      \cdot&\cdots&\cdot&\vline&\cdot&\cdots&\cdot\\
      b_{\hat{s}(p-r+1)}&\cdots&b_{\hat{s}p} &\vline& b_{\hat{s}(\hat{q}-\hat{s}+\hat{1})}&\cdots&b_{\hat{s}\hat{q}}
 \end{pmatrix}  \nonumber
   \end{equation*}
 One can see that
 \begin{equation}
\Ber(A^{-1})_{\Tilde{u}}^{\Tilde{v}}=\Ber T \label{s1}
 \end{equation}
 where $T$ is the super matrix of size $p|q \times p|q$ described as follows:
 \begin{equation*}
   T_{j}=\begin{cases}
       (A^{-1})_{p-r+j} \qquad &\hbox{for} \qquad 1\leq j \leq r  \\
       e_{j} \qquad &\hbox{for} \qquad r+1\leq j \leq p
   \end{cases}
 \end{equation*}
where $T_{j}$ denotes the $j$-th even column of $T$, $(A^{-1})_k$ is the $k$-th even column of $A^{-1}$ and $e_{j}$ denotes an even vector of size $p|q$, whose  entries are all $0$ except at the $j$-th even position, where the entry is $1$. On the other hand,
 \begin{eqnarray*}
   T_{\hat{l}}=\begin{cases}
       (A^{-1})_{\hat{q}-\hat{s}+\hat{l}} \qquad &\hbox{for} \qquad \hat{1}\leq \hat{l} \leq \hat{s}  \\
       e_{\hat{l}} \qquad &\hbox{for} \qquad \hat{s}+\hat{1}\leq \hat{l} \leq\hat{ q},
   \end{cases}
 \end{eqnarray*}
 where $T_{\hat{l}}$ denotes the $\hat{l}$-th odd column of $T$, $(A^{-1})_{\hat{k}}$ is the $\hat{k}$-th-odd column of $A^{-1}$ and $e_{\hat{l}}$ denotes an odd vector of size $p|q$ whose  entries are all $0$ except at the $\hat{l}$-th odd position, where the entry is 1.

In other words, $T$ is a supermatrix whose first $r$ even columns coincide with the last $r$ even columns of $A^{-1}$, the first $\hat{s}$ odd columns coincide with the last $\hat{s}$ odd columns of $A^{-1}$ and the remaining even and odd columns are $e_{j}$ and $e_{\hat{l}}$, respectively.

Therefore, using Equation (\ref{s1}) and the multiplicative property of Berezinian, we have:
\begin{eqnarray}
\Ber A\,\Ber(A^{-1})_{\Tilde{u}}^{\Tilde{v}}=\Ber A\,\Ber T=\Ber(AT). \label{s2}
\end{eqnarray}
Let us denote $C=AT$; then, using above notation:
\begin{eqnarray*}
   C_{j}=\begin{cases}
       e_{p-r+j} \qquad &\hbox{for} \qquad 1\leq j \leq r,  \\
       A_{j} \qquad &\hbox{for} \qquad r+1\leq j \leq p,
   \end{cases}
 \end{eqnarray*}
 and
 \begin{eqnarray*}
   C_{\hat{l}}=\begin{cases}
       e_{{\hat{q}-\hat{s}+\hat{l}}}  \qquad &\hbox{for} \qquad \hat{1}\leq \hat{l} \leq \hat{s}  \\
       A_{\hat{l}}\qquad &\hbox{for} \qquad \hat{s}+\hat{1}\leq \hat{l} \leq \hat{q}.
   \end{cases}
 \end{eqnarray*}
Then, one can compute:
\begin{eqnarray}
\Ber (AT) =\Ber C =(-1)^{r(p+1)+s(q+1)}\Ber A^{u}_{v}
\end{eqnarray}
and by putting this into Equation (\ref{s2}) the desired identity is obtained.
\end{proof}

\section{The Super Pl\"ucker Relations for $\mathrm{SL}(1|1)$} \label{pl-11}
We want to use  the super Jacobi identity as in Theorem (\ref{superjacobi}) to derive the relations
among the generators in the superalgebra $\Tilde{\mathcal{O}}$ defined in (\ref{otilde}). This approach for the construction of classical Pl\"ucler relations has been explained in the Appendix \ref{jacobi-app}, Proposition \ref{P1}.

We illustrate in detail in this section the case of $r=s=1$, the general case
will be discussed shortly, in the next section, following the same strategy.

\begin{theorem}{\sl The super Pl\"ucker relations for $\mathrm{SL}(1|1)$}.
\label{super-pl-thm}
Let the notation be as in (\ref{superminor1}), (\ref{superminor2}) and (\ref{superminor3})
for $r=s=1$. Then:
\begin{align}
&X_{i|\hat{\mu}}X_{i|\hat{\mu}}^{*}=1, \label{sp1}
\\[0.3cm]
  & X_{i|\hat{\mu}}X_{j|\hat{\nu}}= \Ber \begin{pmatrix}
       X_{j|\hat{\mu}} &\vline& X_{\hat{\nu}|\hat{\mu}}\\
       \hline
       X_{i|j}^{*} &\vline& X_{i|\hat{\nu}}^{*}
   \end{pmatrix}, \label{sp2}
\\[0.3cm]
&X_{i|\hat{\mu}}X_{\hat{\lambda}|\hat{\nu}}= \Ber \begin{pmatrix}
       X_{\hat{\lambda}|\hat{\mu}} &\vline& X_{\hat{\nu}|\hat{\mu}}\\
       \hline
       X_{i|\hat{\lambda}}^{*} &\vline& X_{i|\hat{\nu}}^{*}
   \end{pmatrix}, \label{sp3}
\\[0.3cm]
&
X_{i|\hat{\mu}}^{*}X_{j|k}^{*}= \Ber^{*} \begin{pmatrix}
       X_{j|\hat{\mu}} &\vline& X_{k|\hat{\mu}}\\
       \hline
       X_{i|j}^{*} &\vline& X_{i|k}^{*}
   \end{pmatrix}, \label{sp4}
\end{align}
where the Latin (even) indices run from 1 to $p$ and the Greek (odd) indices run from 1 to $q$.
\end{theorem}

We call these the \textit{super Pl\"ucker relations} for $\mathrm{SL}(1|1)$. Moreover, it is important to note that these super Pl\"ucker relations coincides with the super Pl\"ucker relations appeared in \cite[Th. 8]{Voronov} obtained with a different approach. 

\begin{proof}
The relations (\ref{sp1}) follow from the Equation (\ref{ob}).

Now, for any indices $(i|\hat{\mu})$ and $(j|\hat{\nu})$, fix the following matrix:
\begin{eqnarray*}
B=\begin{pmatrix}
 x_{1i} & x_{1j} & \vline& x_{1\hat{\mu}} & x_{1\hat{\nu}}\\
 0 & 1 & \vline& 0 & 0\\
 \hline
 x_{\hat{1}i} & x_{\hat{1}j} &\vline& x_{\hat{1}\hat{\mu}} & x_{\hat{1}\hat{\nu}}\\
 0 & 0 &\vline& 0 & 1
\end{pmatrix}.
\end{eqnarray*}
Clearly,
\begin{eqnarray}
\Ber B=X_{i|\hat{\mu}}    \label{SFT11}
\end{eqnarray}
By setting $u=(2,\hat{2})$ and $v=(1,\hat{1})$ we get:
\begin{eqnarray}
 \Ber B_{v}^{u}=\Ber \begin{pmatrix}
     x_{1j} &\vline& x_{1\hat{\nu}}\\
     \hline
     x_{\hat{1}j} &\vline& x_{\hat{1}\hat{\nu}}
 \end{pmatrix}= X_{j|\hat{\nu}}.  \label{SFT12}
\end{eqnarray}
Moreover, since $\tilde u=(1|\hat1)$ and  $\tilde v=(2|\hat 2)$, using the super Cramer rule one can compute:
\begin{equation*}
(B^{-1})^{\Tilde{v}}_{\Tilde{u}}=\begin{pmatrix}
    (B^{-1})_{12} &\vline& (B^{-1})_{1\hat{2}}\\
    \hline
    (B^{-1})_{\hat{1}2} &\vline& (B^{-1})_{\hat{1}\hat{2}}
\end{pmatrix}=
\begin{pmatrix}
    -X_{j|\hat{\mu}}X^{*}_{i|\hat{\mu}} &\vline& -X_{\hat{\nu}|\hat{\mu}}X^{*}_{i|\hat{\mu}}\\
    \hline
    -X^{*}_{i|j} X_{i|\hat{\mu}} &\vline& -X^{*}_{i|\hat{\nu}} X_{i|\hat{\mu}}
\end{pmatrix}.
\end{equation*}
Hence,
\begin{eqnarray}
\Ber (B^{-1})^{\Tilde{v}}_{\Tilde{u}}= (X^{*}_{i|\hat{\mu}})^{2}\Ber \begin{pmatrix}
    X_{j|\hat{\mu}} &\vline& X_{\hat{\nu}|\hat{\mu}}\\
    \hline
    X^{*}_{i|j} &\vline& X^{*}_{i|\hat{\nu}}
\end{pmatrix}. \label{SFT13}
\end{eqnarray}
Substituting Equations (\ref{SFT11}), (\ref{SFT12}) and (\ref{SFT13}) in the super Jacobi identity, one arrives to the relations (\ref{sp2}).

For the  other relations, we would need to generalize the super Jacobi identity to include some fake supermatrices. We instead use an argument, similar to the one used in the proof of the super Jacobi identity.  Consider the matrix
\begin{eqnarray}
B=\begin{pmatrix}
 x_{1i} & 0 & \vline& x_{1\hat{\mu}} & x_{1\hat{\nu}}\\
 0 & 1 & \vline& 0 & 0\\
 \hline
 x_{\hat{1}i} & 0 &\vline& x_{\hat{1}\hat{\mu}} & x_{\hat{1}\hat{\nu}}\\
 0 & 0 &\vline& 0 & 1
\end{pmatrix}
\end{eqnarray}
with $\Ber B=X_{i|\hat \mu}$. We construct a fake supermatrix $T$ in the following way: the first even column is the second even column of $B^{-1}$, the first odd column is the second  odd column of $B^{-1}$; at the second even column position we have $B^{-1}t$  where

$$t:=\begin{pmatrix}
        x_{1\hat{\lambda}}\\
        0\\
        \hline
        x_{\hat{1}\hat{\lambda}}\\
        0
    \end{pmatrix}.$$
and at the second odd column position we have  $e_{\hat{2}}$.
Then, clearly:
\begin{eqnarray}
    \Ber(BT)=\Ber \begin{pmatrix}
 0 &  x_{1\hat{\lambda}} & \vline& 0 & x_{1\hat{\nu}}\\
 1 & 0 & \vline& 0 & 0\\
 \hline
 0 & x_{\hat{1}\hat{\lambda}} &\vline& 0 & x_{\hat{1}\hat{\nu}}\\
 0 & 0 &\vline& 1 & 1
\end{pmatrix}=X_{\hat{\lambda}|\hat{\nu}}. \label{SFT14}
\end{eqnarray}

On the other hand, one can explicitly compute, using super Cramer's rule:
\begin{eqnarray}
    T=\begin{pmatrix}
    0 &  X_{\hat{\lambda}|\hat{\mu}}X^{*}_{i|\hat{\mu}} & \vline& -X_{\hat{\nu}|\hat{\mu}}X^{*}_{i|\hat{\mu}} & 0\\
 1 & 0 & \vline& 0 & 0\\
 \hline
 0 & X^{*}_{i|\hat{\lambda}} X_{i|\hat{\mu}} &\vline& -X^{*}_{i|\hat{\nu}}X_{i|\hat{\mu}} & 0\\
 0 & 0 &\vline& 1 & 1
    \end{pmatrix}
\end{eqnarray}
Therefore,
\begin{eqnarray}
\Ber T= (X^{*}_{i|\hat{\mu}})^{2} \Ber \begin{pmatrix}
    X_{\hat{\lambda}|\hat{\mu}} &\vline& X_{\hat{\nu}|\hat{\mu}}\\
    \hline
    X^{*}_{i|\hat{\lambda}} &\vline& X^{*}_{i|\hat{\nu}}
\end{pmatrix}.
\end{eqnarray}
Using the multiplicative property of the Berezinian and equation (\ref{SFT14}) one arrives to the relations (\ref{sp3}).

Similarly, the relations (\ref{sp4}) follow. Consider the supermatrix
$$B=\begin{pmatrix}x_{\hat{1}\hat{\mu}}&0&\vline&x_{\hat{1}i}&x_{\hat{1}j}\\
0&1&\vline&0&0\\\hline
x_{{1}\hat{\mu}}&0&\vline&x_{{1}i}&x_{{1}j}\\
0&0&\vline&0&1
\end{pmatrix},$$ whose Berezinian is $\Ber B= X_{i|\hat{\mu}}^*$. Define a the supermatrix $T$ whose first even column is the second even column of $B^{-1}$ and whose second even column is $B^{-1}t$ with
$$t=\begin{pmatrix}x_{\hat{1}k}\\0\\\hline x_{1k}\\0\end{pmatrix};$$ on the other hand, its first odd column is the second odd column of $B^{-1}$ and its second odd column is $e_{\hat{2}}$. Then
$$BT=\begin{pmatrix}0&x_{\hat{1}k}&\vline&0&x_{\hat{1}j}\\
1&0&\vline&0&0\\\hline
0&x_{1k}&\vline&0&x_{1j}\\
0&0&\vline&1&1\end{pmatrix}.$$
The Berezinian $\Ber (BT)=X^*_{j|k}$. As before, we can compute $T$ explicitly with the help of the super Cramer rule:
$$T=\begin{pmatrix}0&X^*_{i|k}X_{i|\hat{\mu}}&\vline&-X^*_{i|j} X_{i|\hat{\mu}} &0\\
1&0&\vline&0&0\\\hline
0&X_{k|\hat{\mu}}X^*_{i|{\mu}}&\vline&-X_{j|\hat{\mu}}X^*_{i|\hat{\mu}}&0\\
0&0&\vline&0&1
\end{pmatrix}.$$ Then
$$\Ber T=X_{i|\hat{\mu}}^2\Ber\begin{pmatrix}X^*_{i|\hat{\mu}}&\vline&X^*_{i|j}\\\hline X_{k|\hat{\mu}}&\vline&X_{j|\hat{\mu}}\end{pmatrix},$$
so $$\Ber(BT)=\Ber B\,\Ber T=X^*_{i|\hat{\mu}}X_{i|\hat{\mu}}^2\Ber^*\begin{pmatrix}X_{j|\hat{\mu}}&\vline&X^*_{k|\hat{\mu}}\\\hline X_{i|j}&\vline&X^*_{i|k}\end{pmatrix}.$$ We then arrive to (\ref{sp4}), as we wanted to show.
\end{proof}

\section{Second Fundamental Theorem of super invariant theory for
$\mathrm{SL}(1|1)$}\label{sec-sft}

Now, we want to prove that the relations that appeared in Equations
(\ref{sp1}), (\ref{sp2}), (\ref{sp3}), and (\ref{sp4}) completely characterize the
super ring of invariants $\Tilde{\mathcal{O}}^{\mathrm{SL}(1|1)}$. In this way, we prove the Second Fundamental Theorem of super invariant theory for $\mathrm{SL}(1|1)$. 

Let us recall the following basic fact from linear algebra.

Let $\{v_{1}, v_{2}, \cdots \}$ be an ordered (countable) basis of a vector space $V$.
For any non-zero vector $v=\sum_{i}a_{i}v_{i}$, define the \textit{leading term} of $v$
as $\mathrm{Lt}(v):=v_{i_{0}}$, where $a_{i_{0}}\neq 0$, and $a_{i}=0$ for all $v_{i}< v_{i_{0}}$.

\begin{lemma}\label{linear-ind}
Any collection of vectors $\{v_{1}, v_{2}, \dots \}$ in $V$, with distinct leading
terms with respect to an ordered basis, as above, is linearly independent.
 \end{lemma}

The superalgebra $\Tilde{\mathcal{O}}^{\mathrm{SL}(1|1)}$ is generated by $X_{i|\hat{\mu}}$, $X_{\hat{\mu}| \hat{\nu}}$, $X_{j|\hat{\nu}}^{*}$, and $X_{i|j}^{*}$. Each element of $\Tilde{\mathcal{O}}^{\mathrm{SL}(1|1)}$ is a linear combination of expressions of the following form:
\begin{eqnarray} \label{expr}
\underset{i|\hat{\mu}, j|\hat{\lambda}, \hat{\eta}|\hat{\nu}, k|l }{\Pi} X_{i|\hat{\mu}}^{d_{i|\hat{\mu}}} X_{\hat{\eta}| \hat{\nu}}^{d_{\hat{\eta}|\hat{\nu}}}X^{*d_{j|\hat{\lambda}}}_{j|\hat{\lambda}}X^{*d_{k|l}}_{k|l},
\end{eqnarray}
where $d_{i|\hat{\mu}}, d_{j|\hat{\lambda}}\in \{0,1,2,...\}$ and $d_{\hat{\eta}|\hat{\nu}}, d_{k|l}\in \{0,1\}$. However, due to super Pl\"ucker relations (\ref{sp1}), (\ref{sp2}), (\ref{sp3}) and (\ref{sp4}), we can express:
\begin{eqnarray}
X_{i|\hat{\mu}}&=& X^*_{1|\hat{1}}X_{1|\hat{\mu}}X_{i|\hat{1}} -X^*_{1|\hat{1}}X_{\hat{\mu}|\hat{1}}X_{1|\hat{\mu}}^{2}X_{1|i}^{*}, \label{spp1}\\
X^{*}_{j|\hat{\lambda}}&=& X_{1|\hat{1}}X^*_{j|\hat{1}}X^*_{1|\hat{\lambda}} -X_{1|\hat{1}}X_{1|j}^{*}X^{*2}_{j|\hat{1}}X_{\hat{\lambda}|\hat{1}},\label{spp2}\\
X_{\hat{\eta}|\hat{\nu}}&=& X^*_{1|\hat{1}}X_{1|\hat{\nu}}X_{\hat{\eta}|\hat{1}} -X^*_{1|\hat{1}}X_{\hat{\nu}|\hat{1}}X_{1|\hat{\nu}}^{2}X_{1|\hat{\eta}}^{*},\label{spp3}\\
X^*_{k|l}&=& X_{1|\hat{1}}X^*_{k|\hat{1}}X^*_{1|l} -X_{1|\hat{1}}X^*_{1|k}X^{*2}_{k|\hat{1}}X_{l|\hat{1}}. \label{spp4}
\end{eqnarray}
{ for $1\leq i,j,k,l \leq p$ and $1\leq \mu, \lambda, \nu, \eta \leq q$.}

Therefore, every expression in $(\ref{expr})$ can be written and arranged as a linear combination of  expressions of the following form:
\begin{eqnarray}
P=P_{1|\hat{\underline{\mu}}}P_{\underline{i}|\hat{1}}P_{\hat{\underline{\eta}}|\hat{1}}P^{*}_{1|\hat{\underline{\lambda}}}P^{*}_{\underline{j}|\hat{1}}P^{*}_{1|\underline{l}}, \label{standard}
\end{eqnarray}
(the underlined indices mean multiindices) where:
\begin{eqnarray*}
P_{1|\hat{\underline{\mu}}}:=X^{d_{\hat{\mu}_{1}}}_{1|\hat{\mu}_{1}}\cdots X^{d_{\hat{\mu}_{\alpha}}}_{1|\hat{\mu}_{\alpha}},
 \qquad P_{\underline{i}|\hat{1}}:=X^{d_{i_{1}}}_{i_{1}|\hat{1}}\cdots X^{d_{i_{a}}}_{i_{a}|\hat{1}},
\qquad P_{\hat{\underline{\eta}}|\hat{1}}:=X_{\hat{\eta}_{1}|\hat{1}}\cdots X_{\hat{\eta}_{\beta}|\hat{1}},\\
P^{*}_{1|\hat{\underline{\lambda}}}:=X^{*d_{\hat{\lambda}_{1}}}_{1|\hat{\lambda}_{1}}\cdots X^{*d_{\hat{\lambda}_{\gamma}}}_{1|\hat{\lambda}_{\gamma}}, \qquad P^{*}_{\underline{j}|\hat{1}}:=X^{*d_{j_{1}}}_{j_{1}|\hat{1}}\cdots X^{*d_{j_{b}}}_{j_{b}|\hat{1}}, \qquad P^{*}_{1|\underline{l}}:= X^{*}_{1|l_{1}}\cdots X^{*}_{1|l_{c}},
\end{eqnarray*}
and
$$1 \leq \mu_{1}<\cdots < \mu_{\alpha} \leq q, \qquad 1<i_{1} < \cdots < i_{a}\leq p,$$
$$1 < \eta_{1} < \cdots < \eta_{\beta} \leq q, \qquad 1 \leq \lambda_{1} < \cdots < \lambda_{\gamma} \leq q,$$
$$ 1<j_{1} < \cdots < j_{b}\leq p, \qquad 1<l_{1} < \cdots < l_{c}\leq p,$$
with the conditions 
\begin{itemize}
    \item $\hat{\mu}_{\epsilon} \neq \hat{\lambda}_{\delta}$ for any $ \epsilon, \delta \in \{1,\cdots, q\}$,
    \item $i_{m} \neq j_{n}$ for any $m,n \in \{2, \cdots , p\}$.
\end{itemize}

\medskip
\begin{definition}\label{standard-gen-def}
{ We call the elements of the set:
\begin{eqnarray*}
\{X_{1|\hat{\mu}}, X_{i|\hat{1}}, X^*_{1|\hat{\mu}}, X^*_{i|\hat{1}}, X_{\hat{\lambda}|\hat{1}}, X^*_{1|j}| 1 \leq i,j \leq p, ~ 2 \leq \hat{\mu}, \hat{\lambda} \leq q\}
\end{eqnarray*}
\textit{standard expressions}. The products of the standard expressions that appear in $P$, as in 
(\ref{standard}), are \textit{standard products}.     }
\end{definition}
\medskip

\begin{lemma}\label{basis-gen}
The standard products are linearly independent and form a vector space basis of $\Tilde{\mathcal{O}}^{\mathrm{SL}(1|1)}$.
\end{lemma}

\begin{proof}
{The standard products generate $\Tilde{\mathcal{O}}^{\mathrm{SL}(1|1)}$,
hence we only need to show that they are linearly independent.} 

Let $\mathcal{B}$ be the monomial basis of $\Tilde{\mathcal{O}}$, with expressions of the following form:
$$\Pi_{(i,j)}x_{ij}^{a_{ij}}(x_{ij}^{-1})^{b_{ij}} \Pi_{(k,l)} x_{kl}^{c_{kl}},$$
where:
\begin{itemize}
\item $x_{ij}$ and $x_{kl}$ denotes even and odd generators respectively,
\item $a_{ij}, b_{ij} \in \{0,1,2, \cdots \}$, $c_{kl} \in\{0,1\}$,
    \item for a fixed $(i,j)$ at least one of $a_{ij}$ or $b_{ij}$ is zero.
\end{itemize}
Now, consider the following ordering of tuples,
$$(1,1)< \cdots < (1,p) < (\hat{1},\hat{1})< \cdots <(\hat{1}, \hat{q})<$$
$$(1,1)^{-1}<  \cdots < (1,p)^{-1} < (\hat{1},\hat{1})^{-1}< \cdots <(\hat{1}, \hat{q})^{-1}<$$
$$(1, \hat{1})< \cdots < (1, \hat{q}) < (\hat{1},1) < \cdots < (\hat{1},p),$$
where $(i,j)^{-1}$ is just a notation for denoting the generators $x_{ij}^{-1}$. This gives us a total ordering on the basis $\mathcal{B}$ of $\Tilde{\mathcal{O}}$ as follows:
$$\Pi_{(i,j)}x_{ij}^{a_{ij}}(x_{ij}^{-1})^{b_{ij}} \Pi_{(k,l)} x_{kl}^{c_{kl}}  < \Pi_{(i,j)}x_{ij}^{a'_{ij}}(x_{ij}^{-1})^{b'_{ij}} \Pi_{(k,l)} x_{kl}^{c'_{kl}}$$
if one of the following holds:\\
\medskip

(i) $\exists$ an $(i,j)$ for which $a_{ij}\neq a'_{ij}$ and $(i,j)$ is smallest such tuple, then $a_{ij}>a'_{ij}$.\\
\medskip

(ii) $a_{ij}=a'_{ij}$ for all $(i,j)$, and $\exists$ an $(i,j)$ for which $b_{ij}\neq b'_{ij}$ and $(i,j)$ is smallest such tuple then, $b_{ij}<b'_{ij}$.\\
\medskip

(iii) $a_{ij}=a'_{ij}$ and $b_{ij}=b'_{ij}$ for all $(i,j)$, and $(k,l)$ is smallest tuple for which $c_{kl}\neq c'_{kl}$, then $c_{kl}> c'_{kl}$.\\
\medskip

With this ordering the leading terms of $P_{1|\hat{\underline{\mu}}}$, $P_{\underline{i}|\hat{1}}$, $P_{\hat{\underline{\eta}}|\hat{1}}$, $P^{*}_{1|\hat{\underline{\lambda}}}$, $P^{*}_{\underline{j}|\hat{1}}$, $P^{*}_{1|\underline{l}}$ turns out to be:
\begin{align*}
&\mathrm{Lt}(P_{1|\hat{\underline{\mu}}})=x_{11}^{d_{\underline{\mu}}}(x_{\hat{1}\hat{\mu_{1}}}^{-1})^{d_{\hat{\mu}_{1}}}...(x_{\hat{1}\hat{\mu_{\alpha}}}^{-1})^{d_{\hat{\mu}_{\alpha}}}, \qquad &\mathrm{Lt}(P_{\underline{i}|\hat{1}})=x_{1i_{1}}^{d_{i_{1}}}...x_{1i_{a}}^{d_{i_{a}}}(x_{\hat{1},\hat{1}}^{-1})^{d_{\underline{i}}},\\
&\mathrm{Lt}(P_{\hat{\underline{\eta}}|\hat{1}})=(x_{\hat{1}\hat{1}}^{-1})^{\beta}x_{1\hat{\eta}_{1}}...x_{1\hat{\eta}_{\beta}}, \qquad &\mathrm{Lt}(P^{*}_{1|\hat{\underline{\lambda}}})=x_{\hat{1}\hat{\lambda}_{1}}^{d_{\hat{\lambda}_{1}}}...x_{\hat{1}\hat{\lambda}_{\gamma}}^{d_{\hat{\lambda}_{\gamma}}}(x_{11}^{-1})^{d_{\hat{\underline{\lambda}}}},\\
&\mathrm{Lt}(P^{*}_{\underline{j}|\hat{1}})=x_{\hat{1}\hat{1}}^{d_{\underline{j}}}(x_{1j_{1}}^{-1})^{d_{j_{1}}}...(x_{1j_{1}}^{-1})^{d_{j_{b}}}, \qquad &\mathrm{Lt}(P^{*}_{1|\underline{l}})= (x_{11}^{-1})^{c}x_{\hat{1}l_{1}}...x_{\hat{1}l_{c}},
\end{align*}
where $d_{\underline{\mu}}:=d_{\mu_{1}}+ \cdots d_{\mu_{\alpha}}$ and similarly others. Moreover, the leading term $\mathrm{Lt}(P)$ of the standard generator $P=P_{1|\hat{\underline{\mu}}}P_{\underline{i}|\hat{1}}P_{\hat{\underline{\eta}}|\hat{1}}P^{*}_{1|\hat{\underline{\lambda}}}P^{*}_{\underline{j}|\hat{1}}P^{*}_{1|\underline{l}}$ is the product of these leading terms. Therefore, distinct standard generators have distinct leading terms, and, by Lemma \ref{linear-ind}, they are linearly independent.
\end{proof}
\medskip

The Lemma \ref{basis-gen} implies that the super Pl\"ucker relations are all the algebraic relations among the generators. We state this precisely in the theorem below.

\begin{theorem}
{\sl Second Fundamental Theorem of super invariant theory for $\mathrm{SL}(1|1)$.}\label{SecFT11}
We have a presentation of the ring of invariants,
\begin{eqnarray*}
\Tilde{\mathcal{O}}^{\mathrm{SL}(1|1)} \cong \mathbb{C}[Y_{i|\hat{\mu}}, Y_{\hat{\mu}| \hat{\nu}},Y_{j|\hat{\nu}}^{*},Y_{i|j}^{*}]/ \mathrm{I}_{sp}
\end{eqnarray*}
where $Y_{i|\hat{\mu}}$ and $Y_{j|\hat{\nu}}^{*}$ are even variables, $Y_{\hat{\mu}| \hat{\nu}}$ and $Y_{i|j}^{*}$ are odd variables, and $\mathrm{I}_{sp}$ is the ideal generated by super Pl\"ucker relations given in Equations (\ref{sp1}), (\ref{sp2}), (\ref{sp3}) and (\ref{sp4}).
\end{theorem}
\begin{proof}
Define,
\begin{eqnarray}
 &\pi : \mathbb{C}[Y_{i|\hat{\mu}}, Y_{\hat{\mu}| \hat{\nu}},Y_{j|\hat{\nu}}^{*},Y_{i|j}^{*}] \longrightarrow \Tilde{\mathcal{O}}^{\mathrm{SL}(1|1)}& \nonumber\\
 &Y_{i|\hat{\mu}}, Y_{\hat{\mu}| \hat{\nu}},Y_{j|\hat{\nu}}^{*},Y_{i|j}^{*} \mapsto X_{i|\hat{\mu}}, X_{\hat{\mu}| \hat{\nu}},X_{j|\hat{\nu}}^{*},X_{i|j}^{*}.&
\end{eqnarray}
Then, clearly, $\mathrm{I}_{sp} \subset \mathrm{ker} (\pi)$.
We need to show that, in fact, $\mathrm{I}_{sp} = \mathrm{ker} (\pi)$.

{ Assume $f \in \mathrm{ker} (\pi)$. By
Lemma \ref{basis-gen} we can write:
$$
f=  \sum_i a_{i} P_{i}(Y)+Q
$$
where the $P_i(Y)$ are standard products in $Y$'s and $Q\in \mathrm{I}_{sp}$. We have
$$
\pi(f)=\sum_i a_i P_i(X)=0
$$
where $P_i(X)$ are standard products in $X$'s. By Lemma \ref{basis-gen}, we have $a_i=0$
for all $i$, so that $f=Q\in \mathrm{I}_{sp}$.
}
\end{proof}

\section{Super Pl\"ucker relations for $\mathrm{SL(r|s)}$}
\label{sft-sec}

In this section, we prove one of the main results of our paper: the generalization
to the super setting of the Pl\"ucker relations for $\mathrm{SL(r|s)}$, expressed
for the ordinary setting in Thm. \ref{fsft}. Again, it is important to note that these super Pl\"ucker relations in the general case also coincides with the relations appeared in \cite[Th. 9]{Voronov} obtained with a different approach.

\begin{theorem}
\label{super-pl}
{Let the notation be as in (\ref{superminor1}), (\ref{superminor2}) and (\ref{superminor3}).
Then,}
\begin{eqnarray}
X_{\underline{i}|\underline{\hat{\mu}}}X^{*}_{\underline{i}|\underline{\hat{\mu}}}&=& 1 \label{gsp1}
\end{eqnarray}

\begin{eqnarray}
(X_{\underline{i}|\underline{\hat{\mu}}})^{r+s-1}X_{\underline{j}|\underline{\hat{\nu}}} &=& \Ber \begin{pmatrix}
     X_{\underline{i}_{a}(j_{t})|\underline{\hat{\mu}}}  &\vline& X_{\underline{i}_{a}(\nu_{\beta})|\underline{\hat{\mu}}}\\
        \hline
       X^{*}_{\underline{i}|\underline{\hat{\mu}}_{\alpha}(j_{t})}  &\vline& X^{*}_{\underline{i}|\underline{\hat{\mu}}_{\alpha}(\nu_{\beta})}
\end{pmatrix}_{\substack{a,t=1,...,r \\ \alpha, \beta =1,...,s}} \label{gsp2}
\end{eqnarray}

\begin{eqnarray}
(X_{\underline{i}|\underline{\hat{\mu}}})^{r+s-1}X_{j_{1}...j_{r-1} \hat{\lambda} | \underline{\hat{\nu}}}&=& \Ber \begin{pmatrix}
     X_{\underline{i}_{a}(j_{t})|\underline{\hat{\mu}}} & X_{\underline{i}_{a}(\hat{\lambda})|\underline{\hat{\mu}}} &\vline& X_{\underline{i}_{a}(\nu_{\beta})|\underline{\hat{\mu}}}\\
        \hline
       X^{*}_{\underline{i}|\underline{\hat{\mu}}_{\alpha}(j_{t})} & X^{*}_{\underline{i}|\underline{\hat{\mu}}_{\alpha}(\hat{\lambda})}  &\vline& X^{*}_{\underline{i}|\underline{\hat{\mu}}_{\alpha}(\nu_{\beta})}
\end{pmatrix}_{\substack{a=1,...,r \\ t=1,...,r-1\\ \alpha, \beta =1,...,s}} \label{gsp3}\\
(X^{*}_{\underline{i}|\underline{\hat{\mu}}})^{-r-s+1}X^{*}_{\underline{j}|\nu_{1}...\nu_{s-1}y}&=& \Ber^{*}  \begin{pmatrix}
     X_{\underline{i}_{a}(j_{t})|\underline{\hat{\mu}}}  &\vline& X_{\underline{i}_{a}(\nu_{\beta})|\underline{\hat{\mu}}} & X_{\underline{i}_{a}(y)|\underline{\hat{\mu}}}\\
        \hline
       X^{*}_{\underline{i}|\underline{\hat{\mu}}_{\alpha}(j_{t})}  &\vline& X^{*}_{\underline{i}|\underline{\hat{\mu}}_{\alpha}(\nu_{\beta})} & X^{*}_{\underline{i}|\underline{\hat{\mu}}_{\alpha}(y)}
\end{pmatrix}_{\substack{a,t=1,...,r \\ \alpha=1,...,s \\ \beta =1,...,s-1}} \label{gsp4}
\end{eqnarray}
where $\underline{i}_{a}(j_{t})|\underline{\mu}$ denotes replacing the a-th index of $\underline{i}=i_{1},...,i_{r}$ with $j_{t}$ and similarly others. We call these the \textit{super Pl\"ucker relations} for $\mathrm{SL}(r|s)$.
\end{theorem}

\begin{proof}
  Consider the ordered multi-indices $\underline{i}|\underline{\hat{\mu}}=(i_{1},...,i_{r}|\hat{\mu}_{1},...,\hat{\mu}_{s})$ and $\underline{j}|\underline{\hat{\nu}}=(j_{1},...,j_{r}|\hat{\nu}_{1},...,\hat{\nu}_{s})$ taken from $(1,\dots,p|\hat{1},\dots,\hat{q})$. We fix the following matrix of size $(2r|2s) \times (2r|2s)$:
\begin{eqnarray*}
A=\begin{pmatrix}
  A_{1} & A_{2} &\vline& A_{3} & A_{4}\\
   O   & I_{r} &\vline& O & O\\
   \hline
  A_{5} & A_{6} &\vline& A_{7} & A_{8}\\
  O   & O &\vline& O & I_{s}\\
\end{pmatrix}.
\end{eqnarray*}
where $O$ and $I$ denote the null and identity matrices, respectively, and
\begin{align*}
&A_{1}=\begin{bmatrix}
        x_{ki_{l}}
    \end{bmatrix}\qquad &A_{2}&=\begin{bmatrix}
        x_{kj_{l}}
    \end{bmatrix} \qquad
    &A_{3}=\begin{bmatrix}
        x_{k\hat{\mu}_{n}}
    \end{bmatrix} \qquad &A_{4}&=\begin{bmatrix}
        x_{k\hat{\nu}_{n}}
    \end{bmatrix}  \\
    &A_{5}=\begin{bmatrix}
        x_{\hat{m}i_{l}}
    \end{bmatrix} \qquad &A_{6}&=\begin{bmatrix}
        x_{\hat{m}j_{l}}
    \end{bmatrix} \qquad
    &A_{7}=\begin{bmatrix}
        x_{\hat{m}\hat{\mu}_{n}}
    \end{bmatrix} \qquad &A_{8}&=\begin{bmatrix}
        x_{\hat{m}\hat{\nu}_{n}}
    \end{bmatrix}
\end{align*}
where $k,l=1,...,r$ and $m,n=1,...,s$. Then, clearly,
\begin{eqnarray}
\Ber(A)=X_{\underline{i}|\underline{\hat{\mu}}}.    \label{gs1}
\end{eqnarray}
Moreover, by setting,
\begin{eqnarray*}
u=(r+1,\hdots ,2r|\hat{s+1},\hdots,\hat{2s}) \quad \text{and} \quad v=(1,\hdots, r| \hat{1},\hdots \hat{s})
\end{eqnarray*}
we get:
\begin{eqnarray}
\Ber(A^{u}_{v})=\Ber \begin{pmatrix}
A_{2} &\vline& A_{4}\\
\hline
A_{6} &\vline& A_{8}
\end{pmatrix}=X_{\underline{j}|\underline{\hat{\nu}}}. \label{gs2}
\end{eqnarray}
On the other hand:
\begin{eqnarray*}
(A^{-1})^{\Tilde{v}}_{\Tilde{u}}= \begin{pmatrix}
        (A^{-1})_{ab} &\vline& (A^{-1})_{a\hat{d}}\\
        \hline
        (A^{-1})_{\hat{c}b} &\vline& (A^{-1})_{\hat{c}\hat{d}}
\end{pmatrix}
\end{eqnarray*}
where $a=1,...,r,$ $b=r+1,...,2r,$ $c=1,...,s$ and $d=s+1,...,2s$.\\

These entries of $A^{-1}$ can be computed using the super Cramer rule. For instance,
\begin{eqnarray*}
(A^{-1})_{ab}=\dfrac{\Ber(A_{a}(e_{b}))}{\Ber(A)}=-X_{\underline{i}_{a}(j_{t})|\underline{\hat{\mu}}}X^{*}_{\underline{i}|\underline{\hat{\mu}}}
\end{eqnarray*}
where $t=b-r$.

Similarly, the other entries are obtained. We get:
\begin{eqnarray}
(A^{-1})^{\Tilde{v}}_{\Tilde{u}}= \begin{pmatrix}
     -X_{\underline{i}_{a}(j_{t})|\underline{\hat{\mu}}}X^{*}_{\underline{i}|\underline{\hat{\mu}}}    &\vline& -X_{\underline{i}_{a}(\nu_{\beta})|\underline{\hat{\mu}}}X^{*}_{\underline{i}|\underline{\hat{\mu}}}\\
        \hline
       -X^{*}_{\underline{i}|\underline{\hat{\mu}}_{\alpha}(j_{t})}X_{\underline{i}|\underline{\hat{\mu}}}  &\vline& -X^{*}_{\underline{i}|\underline{\hat{\mu}}_{\alpha}(\nu_{\beta})}X_{\underline{i}|\underline{\hat{\mu}}}
\end{pmatrix}_{\substack{a,t=1,...,r \\ \alpha, \beta =1,...,s}} \label{gs3}
\end{eqnarray}

Substituting Equations. (\ref{gs1}), (\ref{gs2}) and (\ref{gs3}) into the super Jacobi identity, we get relations in Equation (\ref{gsp2}). Similarly, the other super Pl\"ucker relations in Eqs. (\ref{gsp3}) and (\ref{gsp4}) are obtained by modifying the strategy used in the proof of super Jacobi identity, as we explicitly presented in the previous section for the case of $\mathrm{SL}(1|1)$.
\end{proof}

{We end this section by stating a conjecture.}
{\begin{conjecture}
$\Tilde{\mathcal{O}}^{\mathrm{SL}(r|s)}$ is isomorphic to the superalgebra generated by  $Y_{\underline{i}|\underline{\hat{\mu}}}$, $Y_{\underline{i}_{a}(\hat{\nu})|\underline{\hat{\mu}}}$, $Y^{*}_{\underline{i}|\underline{\hat{\mu}}_{\alpha}(j)}$ and $Y^{*}_{\underline{i}|\underline{\hat{\mu}}}$ subject to the super Pl\"ucker relations given in Equations (\ref{gsp1}), (\ref{gsp2}), (\ref{gsp3}) and (\ref{gsp4}).
\end{conjecture}}

\section*{Acknowledgments}
We thank Prof. Shemyakova and Prof. Voronov for helpful
discussions. R. Fioresi and J. Razzaq want to thank the Departament de F\'{\i}sica Te\`{o}rica,
Universitat de Val\`{e}ncia and the  Department of Mathematics, LUMS
for the hospitality while part of the work
was done. M. A. Lled\'{o} wants to thank Fabit, Universit\`a di
Bologna, for its kind hospitality. This research was supported by GNSAGA-Indam,
PNRR MNESYS, PNRR National Center
for HPC, Big Data and Quantum Computing CUP J33C22001170001, PNNR SIMQuSEC
CUP J13C22000680006. This work is also supported by the Spanish Grants PID2020-116567GB-C21, PID2023-001292-S and CEX2023-001292-S
funded by MCIU/AEI/10.13039/501100011033.
This work was supported by Horizon Europe EU projects MSCA-SE CaLIGOLA, Project ID: 101086123,
MSCA-DN CaLiForNIA, Project ID: 101119552.

\hyphenation{hos-pi-ta-li-ty}

This article is based upon work from COST Action CaLISTA CA21109 supported by COST (European Cooperation in Science and Technology), www.cost.eu.

\section*{Declaration}
This work is pure mathematical and does not have associated data. Moreover, the authors declare no conflict of interest.

\appendix
\section{Supergeometry} \label{super-app} 

We briefly recall some notions regarding supergroups, superspaces
and their $\mathcal{A}$-points. For more details see, for example, Ref.  \cite{ccf},  Chapters 1 and   3 and \cite{fl}, Chapter 1.

Let $\mathbf{k}$ denote the ground field ($\mathbf{k}=\R,\C$). A
\textit{super vector space} is a $\Z_2$-graded
vector space over $\mathbf{k}$ $V=V_0 \oplus V_1$.
We call the pair of numbers $\dim V_0|\dim V_1$ the \textit{super dimension} of $V$ and we define
$\mathbf{k}^{m|n}:=\mathbf{k}^m \oplus \mathbf{k}^n$.

\medskip
We say that $\mathcal{A}$ is a \textit{commutative superalgebra}
if it is a $\Z_2$-graded algebra,
with grading preserving multiplication and
$ab=(-1)^{p(a)p(b)}ba$, for all (homogeneous) $a,b \in \mathcal{A}$, being $p(a), p(b)=0,1$ according to the grading ($p(a)$ is the {\it parity} of an element $a\in \mathcal{A}$).
An important example of commutative superalgebra
is the \textit{superalgebra of polynomials}:
$$
\mathbf{k}[x_1, \dots, x_m, \xi_1, \dots ,\xi_n]:= k[x_1, \dots ,x_m]\otimes
\wedge( \xi_1, \dots ,\xi_n)
$$
where  Latin letters denote the {\it  even} (commuting) variables,
while Greek letters denote the {\it  odd} (anticommuting) ones.

\medskip
For a commutative superalgebra $\mathcal{A}$, we define the free $\mathcal{A}$-module $\mathcal{A}^{m|n}=\mathcal{A} \otimes k^{m|n}$, where
$$
\mathcal{A}^{m|n}_0:=\mathcal{A}_0 \otimes \mathbf{k}^m \oplus \mathcal{A}_1 \otimes \mathbf{k}^n \quad\hbox{and}\quad
\mathcal{A}^{m|n}_1:=\mathcal{A}_0 \otimes \mathbf{k}^n \oplus \mathcal{A}_1 \otimes \mathbf{k}^m
$$
are, respectively, its even and odd parts.

\medskip
Given a commutative superalgebra $\mathcal{A}$,
the \textit{$\mathcal{A}$-points} of the supervector space $V=V_0\oplus V_1$ are:
$$
V(\mathcal{A})=(\mathcal{A} \otimes V)_0 =\mathcal{A}_0 \otimes V_0+\mathcal{A}_1 \otimes V_1
$$
We call the elements in $(\mathcal{A} \otimes V)_0$ {\it even vectors}, while the elements
in $(\mathcal{A} \otimes V)_1$ are {\it odd vectors}.

For example, the $\mathcal{A}$-points of the supervector space $k^{m|n}$ are:
\begin{equation}\label{fopt-svs}
k^{m|n}(\mathcal{A}):=\mathcal{A}_0 \otimes k^m \oplus \mathcal{A}_1 \otimes k^n=
\left\{(a_1, \dots, a_m,\alpha_1,\dots , \alpha_n)\right\}
\end{equation}
Hence $(a_1, \dots, a_m,\alpha_1,\dots , \alpha_n)$ is an even vector while
$(\beta_1, \dots ,\beta_m,b_1,\dots , b_n)$ is an odd one.

\begin{observation}\label{fopt-yo}
We equivalently interpret the $\mathcal{A}$-points as follows:

\begin{align*}
\mathbf{k}^{m|n}(\mathcal{A})&=
\Hom_\mathrm{(salg)}(\mathbf{k}[x_1, \dots ,x_m, \xi_1, \dots ,\xi_n],\mathcal{A})
\end{align*}
where $\mathrm{(salg)}$ denotes the
category of commutative superalgebras. In fact, an expression\hyphenation{ge-ne-ra-tors}
as (\ref{fopt-svs}), gives a morphism $f$ by specifying the images
of the generators $x_i$'s and $\xi_j$'s: $f(x_i)=a_i$, $f(\xi_j)=\alpha_j$
and vice-versa. We say, with an abuse of language, that
$\mathbf{\mathbf{k}}[x_1, \dots ,x_m, \xi_1, \dots ,\xi_n]$ is the superalgebra of `functions'
on the vector superspace $k^{m|n}$. To be precise, $\mathbf{k}[x_1, \dots, x_m, \xi_1, \dots, \xi_n]$
represents $\mathbf{k}^{m|n}$ when we adopt the language of the {\it functor of points},
that we shall not pursue here. The above superalgebra, $\mathbf{k}[x_1, \dots ,x_m, \xi_1 \dots \xi_n]$,
can be seen also as the superalgebra of global sections of the structural sheaf on the
supervariety $\mathbf{k}^{m|n}$. We shall not go further into the theory, the reader can see, for example, \cite{ccf}, Chapter 3.
\end{observation}

We now take into exam a special case,
important in our treatment.
We define the super vector
space of \textit{supermatrices} as:

\begin{align*}
&\rM_{r|s\times p|q}=\left\{ \begin{pmatrix} X & \vline & Y \\ \hline Z & \vline & W \end{pmatrix}
\right\}, \qquad \hbox{with}\\[0.3cm]
&(\rM_{r|s\times p|q})_0=\left\{ \begin{pmatrix} X & \vline& 0 \\ \hline 0 & \vline & W \end{pmatrix}
\right\}, \quad
(\rM_{r|s\times p|q})_1=\left\{ \begin{pmatrix} 0 & \vline & Y \\ \hline Z & \vline & 0 \end{pmatrix}
\right\}
\end{align*}
where $X$, $Y$, $Z$, $W$ are
respectively $r \times p$, $r \times q$, $s \times p$ and
$s \times q$ matrices with entries in $k$. Whenever $p|q=r|s$,
we write just $\rM(r|s)$.

For a commutative superalgebra $\mathcal{A}$, $(\mathcal{A} \otimes \rM_{r|s\times p|q})_0$
are the {\it even supermatrices}, and $(\mathcal{A} \otimes \rM_{r|s\times p|q})_1$
are the {\it odd supermatrices}:
\begin{align*}
&(\mathcal{A} \otimes \rM_{r|s\times p|q})_0 =\left\{
\begin{pmatrix}
[a_{ij}] & \vline & [\alpha_{il}] \\
\hline [\beta_{kj}] & \vline & [b_{kl}]
\end{pmatrix} \right\}
\\[0.3cm]
&(\mathcal{A} \otimes \rM_{r|s\times p|q})_1=
\left\{
\begin{pmatrix}
[\sigma_{ij}] & \vline & [c_{il}] \\
\hline [d_{kj}] & \vline & [\delta_{kl}]%
\end{pmatrix}\right\}
\end{align*}
Notice that the vectors forming the first $p$ columns
of a $r|s\times p|q$ even supermatrix are even, while the vectors forming the last $q$
columns are odd. Vice-versa for odd supermatrices the first $p$ columns
consist of odd vectors and the last $q$ of even ones.

\medskip
The set of $\mathcal{A}$-points of the \textit{general linear supergroup}, denoted, as usual,
$\rGL(r|s)(\mathcal{A})$,
is the group of automorphisms of $\mathcal{A}^{r|s}$. It consists of
all invertible elements in $\rM(r|s)(\mathcal{A})$. They are characterized
by having invertible diagonal blocks (see, for example,  \cite{ccf}, Chapter 1).
One can then define the \textit{Berezinian}:
\begin{align}
\label{bereq}
\text{Ber}
\begin{pmatrix}X & \vline & Y\nonumber \\
\hline Z & \vline & W \end{pmatrix}
= &\mathrm{det}(X- YW^{-1}Z)\mathrm{det}(W)^{-1}=\\&\det(X)\det(W-ZX^{-1}Y)^{-1},
\end{align}
where `det' is the usual determinant.

\section {The super Cramer rule}\label{cramer-app}
We give here a proof of the super Cramer rule, Theorem \ref{supercramer}.

We need first some notation:  for any even vector $\mathbf{b}$ of size $r|s$, we denote by $\mathbf{b}^{(e)}$ and
$\mathbf{b}^{(o)}$ the column vectors consisting of even and odd coefficients
of $\mathbf{b}$, respectively:
  \begin{eqnarray*}
    \mathbf{b}=\begin{pmatrix}
        b_{1}\\
        \cdot \\
        \cdot \\
        \cdot \\
        b_{r}\\
        \hline
        b_{\hat{1}}\\
        \cdot \\
        \cdot \\
        \cdot \\
        b_{\hat{s}}
    \end{pmatrix},~~~~
    \mathbf{b}^{(e)}=\begin{pmatrix}
        b_{1}\\
        \cdot \\
        \cdot \\
        \cdot \\
        b_{r}
    \end{pmatrix}~~~~
    \mathbf{b}^{(o)}=\begin{pmatrix}
    b_{\hat{1}}\\
        \cdot \\
        \cdot \\
        \cdot \\
        b_{\hat{s}}
    \end{pmatrix}
    \end{eqnarray*}

Theorem \ref{supercramer} states:

{\it
Let $M\in \mathrm{GL}(r|s)$ be an invertible even supermatrix and let $\mathbf{b}\in \mathbf{k}^{r|s}$ be an even vector of size $r|s$, then the solution to the equation
$$M\mathbf{x}=\mathbf{b}$$
is:
\begin{align*}
&x_{i}=\dfrac{\Ber M_{i}(\mathbf{b})}{\Ber M}\qquad \hbox{for}\quad i=1,...,r
\\[0.3cm]&
x_{\hat{\jmath}}=\dfrac{\Ber^{*} M_{\hat{\jmath}}(\mathbf{b})}{\Ber^{*}M}\qquad \hbox{for}\quad\hat{\jmath}=\hat{1},...,\hat{s}
\end{align*}
where $M_{i}(\mathbf{b})$ is the supermatrix obtained by replacing $i$-th even column of $M$ with $\mathbf{b}$ and $M_{\hat{\jmath}}(\mathbf{b})$ is the fake supermatrix of type-II obtained by replacing $\hat{\jmath}$-th odd column of $M$ with $\mathbf{b}$.
}

\medskip
\begin{proof}
We have to show that the above statement holds for all $M\in \mathrm{GL}(r|s)$. We will do it in three steps.
\begin{enumerate}
    \item One can decompose any supermatrix $M\in \mathrm{GL}(r|s)$ as a product $M=M_{+}M_{0}M_{-}$ where $M_{+}, M_{0}$ and $M_{-}$ are supermatrices of the following types respectively:
    $$\begin{pmatrix}
        I & \vline& X\\
        \hline
        O &\vline& I
    \end{pmatrix},~~~ \begin{pmatrix}
        V &\vline& O\\
        \hline
        O &\vline& W
    \end{pmatrix},~~~
     \begin{pmatrix}
        I &\vline& O\\
        \hline
        Z &\vline& I
    \end{pmatrix},$$\\
    where $I$ denotes the identity matrix and $O$ denotes the null matrix. For $M=\begin{pmatrix}
        M_{1} &\vline& M_{2}\\
        \hline
        M_{3} &\vline& M_{4}
    \end{pmatrix}$, one gets:
    \begin{align*}
        X&=M_{2}M_{4}^{-1}\\
        V&=M_{1}-M_{2}M_{4}^{-1}M_{3}\\
        W&=M_{4}\\
        Z&=M_{4}^{-1}M_{3}.
    \end{align*}
\item The statement of the theorem holds for supermatrices of the form $C=AB$, where $A$ and $B$ are of the form $M_{0}$ and $M_{-}$ respectively.\\

Let $A=\begin{pmatrix}
    V &\vline& O\\
    \hline
    O &\vline& W
\end{pmatrix}$ and $B=\begin{pmatrix}
    I &\vline& O\\
    \hline
    Z &\vline& I
\end{pmatrix}$ so,

\begin{align*}
C&=\begin{pmatrix}
    V &\vline& O\\
    \hline
    WZ &\vline& W
\end{pmatrix},\qquad \hbox{and}\\\noalign{\vskip 10pt}
C^{-1}&= \begin{pmatrix}
    V^{-1} &\vline& O\\
    \hline
    -ZV^{-1} &\vline& W^{-1}
\end{pmatrix}=\begin{pmatrix}
    \begin{bmatrix}
     \dfrac{\det V_{i}(e_{j})}{\det V}
    \end{bmatrix} & \vline& O\\
    ~&\vline&\\
    \hline
    ~&\vline&\\
    \begin{bmatrix}
    \dfrac{-\sum_{i=1}^{r}z_{\hat{k}i}\det V_{i}(e_{j})}{\det V}
    \end{bmatrix} &\vline& \begin{bmatrix}
    \dfrac{\det W_{k}(e_{l})}{\det W}
    \end{bmatrix}
\end{pmatrix}\end{align*}
where $e_{j}$ denotes the column vector having $1$ at $j$-th place and $0$ at all other places. Therefore,

\begin{eqnarray}
    C^{-1}\mathbf{b}=
\begin{pmatrix}
            \dfrac{\sum_{j=1}^{r}\det V_{1}(e_{j})b_{j}}{\det V}\\
            .\\
            .\\
            .\\
            \dfrac{\sum_{j=1}^{r}\det V_{r}(e_{j})b_{j}}{\det V}\\
            ~\\
            \hline
            ~\\
            -\dfrac{\sum_{i,j=1}^{r} z_{\hat{1}i}\det V_{i}(e_{j})b_{j}}{\det V } +\dfrac{\sum_{{l}=\hat{1}}^{\hat{s}}\det W_{1}(e_{l})b_{l}}{\det W}\\
            .\\
            .\\
            .\\
            -\dfrac{\sum_{i,j=1}^{r} z_{\hat{s}i}\det V_{i}(e_{j})b_{j}}{\det V} +\dfrac{\sum_{{l}=\hat{1}}^{\hat{s}}\det W_{s}(e_{l})b_{l}}{\det W}
\end{pmatrix}. \label{C^{-1}b}
\end{eqnarray}

On the other hand, using the linearity of determinant, we have:
\begin{align}
\dfrac{\Ber C_{i}(\mathbf{b})}{\Ber C}=&\dfrac{\det W^{-1}\det V_{i}(\mathbf{b}^{(e)})}{\det W^{-1}\det V}=\dfrac{\det V_{i}(\mathbf{b}^{(e)})}{\det V }= \nonumber\\[0.3cm]
&\dfrac{\sum_{j=1}^{r}\det V_{i}(e_{j})b_{j}}{\det V} \label{C^{-1}b1}
\end{align}
and
\begin{eqnarray}
\dfrac{\Ber^{*}C_{\hat{k}}(\mathbf{b})}{\Ber^{*}C}&&=\dfrac{\det V^{-1}\det \left(W_{k}(\mathbf{b}^{(o)})-WZV^{-1}O_{k}(\mathbf{b}^{(e)})\right)}{\det V^{-1}\det W}= \nonumber\\[0.3cm]
&& \dfrac{\det W_{k}(\mathbf{b}^{(o)})-\det W_{k}(WZV^{-1}\mathbf{b}^{(e)})}{\det W}= \nonumber\\[0.3cm]
&& \dfrac{\sum_{l=\hat{1}}^{\hat{s}}\det W_{k}(e_{l})b_{l}}{\det W}-\dfrac{\sum_{i,j=1}^{r} z_{\hat{k}i}\det V_{i}(e_{j})b_{j}}{\det V}. \label{C^{-1}b2}
\end{eqnarray}
\hyphenation{com-pa-ring}
The last two equalities follows from the multi-linearity property of determinant and ordinary the  {Cramer rule}, respectively. Hence, comparing Eqs. (\ref{C^{-1}b}), (\ref{C^{-1}b1}) and (\ref{C^{-1}b2}) one realizes that the theorem holds for supermatrices of type $C$.

\item The theorem holds for supermatrices of type $D=EF$ where $F= \begin{pmatrix}
    U &\vline& Y\\
    \hline
    V &\vline& W
\end{pmatrix}$ is an even supermatrix which already satisfies the theorem and $E=\begin{pmatrix}
    I &\vline& X\\
    \hline
    O &\vline& I
\end{pmatrix}$ with $X$ having only one non-zero entry. Without loss of generality, we suppose that the top-left entry $x_{1\hat{1}}\neq 0$.\\

Therefore,

$$D=
\begin{pmatrix}
  \begin{matrix}
  u_{11}+x_{1\hat{1}}v_{\hat{1}1} &\cdots & u_{1r}+x_{1\hat{1}}v_{\hat{1}r}\\
  u_{21} &\cdots & u_{2r}\\
  \cdot  & \cdots  & \cdot \\
  \cdot  & \cdots  & \cdot \\
  u_{r1} & \cdots  & u_{rr}\\
  \end{matrix}
  & \vline &
  \begin{matrix}
  y_{1\hat{1}}+x_{1\hat{1}}w_{\hat{1}\hat{1}} &\cdots & y_{1\hat{s}}+x_{1\hat{1}}w_{\hat{1}\hat{s}}\\
  y_{2\hat{1}} &\cdots & y_{2\hat{s}}\\
  \cdot & \cdots  & \cdot\\
  \cdot & \cdots  & \cdot\\
  y_{r\hat{1}} & \cdots  & y_{r\hat{s}}\\
  \end{matrix}\\
\hline
  & \vline & \\
  V& \vline & W
\end{pmatrix}
$$

Note that,
\begin{eqnarray}
\textbf{x}=D^{-1}\mathbf{b}=F^{-1}(E^{-1}\mathbf{b}). \label{x}
\end{eqnarray}

Fix
$$\mathbf{c}=E^{-1}\mathbf{b}=\begin{pmatrix}
    b_{1}-x_{1\hat{1}}b_{\hat{1}}\\
    \cdot\\
    \cdot\\
    b_{r}\\
    \hline\\
    b_{\hat{1}}\\
    \cdot\\
    \cdot\\
    b_{\hat{s}}
\end{pmatrix}.$$

Since $F$ satisfies the theorem, equation (\ref{x}) implies
\begin{equation}
x_{i}=\dfrac{\Ber F_{i}(\mathbf{c})}{\Ber F}=
\dfrac{
\det\left( U_{i}(\mathbf{c}^{(e)})
-YW^{-1}V_{i}(\mathbf{b}^{(o)})
\right)
}{\det(U-YW^{-1}V)} \label{xi}
\end{equation}
and
\begin{equation}
x_{\hat{\jmath}}=\dfrac{\Ber^{*}F_{\hat{\jmath}}(\mathbf{c})}{\Ber^{*}F}=
\dfrac{\det\left(W_{\hat{\jmath}}(\mathbf{c}^{(o)})-VU^{-1}Y_{\hat{\jmath}}(\mathbf{c}^{(e)})\right)}{\det(W-VU^{-1}Y)}. \label{xj}
\end{equation}

On the other hand,
\begin{align}
    &\dfrac{\Ber D_{i}(\mathbf{b})}{\Ber D}=\nonumber\\[0.3cm]&\dfrac{\det\left((U+XV)_{i}(\mathbf{b}^{(e)})-(Y+XW)W^{-1}V_{i}(\mathbf{b}^{(o)})\right)}{\det(U-YW^{-1}V)}= \nonumber\\[0.3cm]
    &\dfrac{\det\left((U+XV)_{i}(\mathbf{b}^{(e)})-XV_{i}(\mathbf{b}^{(o)})-YW^{-1}V_{i}(\mathbf{b}^{(o)})\right)}{\det(U-YW^{-1}V) }= \nonumber\\[0.3cm]
    &\dfrac{\det\left(U_{i}(\mathbf{c}^{(e)})-YW^{-1}V_{i}(\mathbf{b}^{(o)})\right)}{\det(U-YW^{-1}V)} \label{xi1}
\end{align}
and
\begin{align}
    &\dfrac{\Ber^{*}D_{\hat{\jmath}}(\mathbf{b})}{\Ber^{*}D}=\dfrac{\det \left(W_{\hat{\jmath}}(\mathbf{b}^{(o)})-V(U+XV)^{-1}    (Y+XW)_{\hat{\jmath}}
    (\mathbf{b}^{(e)})\right)}{\det\left(W-V(U+XV)^{-1}(Y+XW)\right)}= \nonumber\\[0.3cm]
       &\dfrac{\det\left(W_{\hat{\jmath}}(\mathbf{b}^{(o)})-V(U^{-1}-U^{-1}XVU^{-1})\left(Y_{\hat{\jmath}}
    (\mathbf{c}^{(e)})+XW_{\hat{\jmath}}(\mathbf{b}^{(o)})\right)\right)}
    {\det\left(W-V(U^{-1}-U^{-1}XVU^{-1})(Y+XW)\right)}= \nonumber\\[0.3cm]
    &\dfrac{\det(W_{\hat{\jmath}}(\mathbf{b}^{(o)})-VU^{-1}Y_{\hat{j}}(\mathbf{c}^{(e)})-VU^{-1}XW_{\hat{\jmath}}(\mathbf{b}^{(o)})
    -VU^{-1}XVU^{-1}Y_{\hat{\jmath}}(\mathbf{c}^{(e)}))}
    {\det(W-VU^{-1}Y-VU^{-1}XW+VU^{-1}XVU^{-1}Y)}= \nonumber\\[0.3cm]
    &\dfrac{\det(I-VU^{-1}X)\det\left(W_{\hat{\jmath}}(\mathbf{c}^{(o)})-VU^{-1}Y_{\hat{\jmath}}(\mathbf{c}^{(e)})\right)}{\det(I-VU^{-1}X)
    \det(W-VU^{-1}Y)} = \nonumber\\[0.3cm]
    &\dfrac{\det\left(W_{\hat{\jmath}}(\mathbf{c}^{(o)})-VU^{-1}Y_{\hat{\jmath}}(\mathbf{c}^{(e)})\right)}{\det(W-VU^{-1}Y)}.\label{xj1}
    \end{align}
By comparing Eqs. (\ref{xi}), (\ref{xj}), (\ref{xi1}) and (\ref{xj1}), one concludes that the matrices of type $F$ satisfy the Theorem.
\end{enumerate}
One can easily observe that every even supermatrix $M \in \mathrm{GL}(r|s)$ is of type $F$ and this completes the proof.
\end{proof}

\begin{remark}\label{cramerfake}
It is important to note that, in the above proof, we never use the fact that $\mathbf{b}$ is even. The same expressions work for the equation $M\mathbf{x}=\mathbf{b}$ if $\mathbf{b}$ is an odd vector of size $r|s$.
\end{remark}

\section{The Jacobi identity and the Pl\"ucker relations}\label{jacobi-app}
We start by recalling a determinant identity due to Jacobi, the \textit{Jacobi complementary minor theorem} \cite{Jacobi} or
the \textit{Jacobi identity}
for short.

\begin{theorem}{\sl The Jacobi complementary minor theorem or Jacobi identity.}\label{J2}
Let $A$ be an invertible $n \times n $ matrix. Fix the two sets of indices:\hyphenation{in-ver-ti-ble}
$$
u=(n-r+1,...,n) \qquad and \qquad v=(1,...,r)
$$
Let $\tilde{u}$ and $\tilde{v}$ denote the complements of $u$ and $v$ respectively
in $(1,\dots, n)$. Then:
\begin{eqnarray} \label{Jacobi1}
    \det A\det(A^{-1})^{\Tilde{v}}_{\Tilde{u}}= (-1)^{r(n+1)} \det A^{u}_{v}
\end{eqnarray}
where $A^{u}_{v}$ denotes the matrix obtained from $A$ by deleting its rows and columns whose indices are contained in $u$ and $v$, respectively.
\end{theorem}

\begin{proof}
We denote
$
A=\begin{bmatrix}
    a_{ij}
\end{bmatrix}_{n \times n}$
and,
$A^{-1}=\begin{bmatrix}
    b_{kl}
\end{bmatrix}_{n \times n}.$
Therefore,
$$(A^{-1})^{\Tilde{v}}_{\Tilde{u}}=\begin{pmatrix}
    b_{1(n-r+1)} &\cdots& b_{1n}\\
    \cdot &\cdots& \cdot \\
    \cdot &\cdots& \cdot. \\
    b_{r(n-r+1)} &\cdots& b_{rn}
\end{pmatrix}.$$
Notice that,
\begin{align}
\det(A^{-1})^{\Tilde{v}}_{\Tilde{u}}=&\det\begin{pmatrix}
    b_{1(n-r+1)} &\cdots & b_{1n}\\
    \cdot  &\cdots & \cdot  \\
    \cdot  &\cdots & \cdot  \\
    b_{r(n-r+1)} &\cdots & b_{rn}
\end{pmatrix}=\nonumber\\[0.3cm]
& \det \begin{pmatrix}
    b_{1(n-r+1)} &\cdots & b_{1n}& 0 &\cdots & 0  \\
    \cdot  &\cdots & \cdot  & \cdot  &\cdots & \cdot  \\
    \cdot  &\cdots & \cdot  & \cdot  &\cdots & \cdot \\
    b_{r(n-r+1)} &\cdots & b_{rn} &  0 &\cdots & 0  \\
     b_{(r+1) (n-r+1)} &\cdots & b_{(r+1)n} & 1 &\cdots & 0  \\
    \cdot  &\cdots & \cdot  & \cdot  &\cdots & \cdot \\
    \cdot  &\cdots & \cdot  & \cdot  &\cdots & \cdot \\
     b_{n(n-r+1)} &\cdots & b_{nn}& 0 &\cdots & 1
\end{pmatrix}_{n \times n} \label{J1}
\end{align}
Let us denote the $n \times n$ matrix that appeared in Equation (\ref{J1}) by $T$. Hence, using the multiplicative property of the determinant we get:
\begin{align}
 \det A\det(A^{-1})^{\Tilde{v}}_{\Tilde{u}}=& \det A\det T =\det(AT)= \nonumber\\[0.4cm]
 & \det \begin{pmatrix}
    0 &\cdots & 0 & a_{1 (r+1)} &\cdots & a_{1n}\\
    \cdot &\cdots & \cdot & \cdot &\cdots & \cdot \\
    \cdot  &\cdots & \cdot  & \cdot  &\cdots & \cdot \\
    0 &\cdots & 0 & a_{(n-r)(r+1)} &\cdots & a_{(n-r)n}\\
    1 &\cdots & 0 & a_{(n-r+1)r+1} &\cdots & a_{(n-r+1)n}\\
    \cdot  &\cdots & \cdot  & \cdot  &\cdots & \cdot \\
    \cdot  &\cdots & \cdot  & \cdot  &\cdots & \cdot \\
    0 &\cdots & 1 & a_{n(r+1)} &\cdots & a_{nn}
\end{pmatrix}=\nonumber\\[0.4cm]
&(-1)^{r(n+1)}\det \begin{pmatrix}
a_{1 (r+1)} &\cdots & a_{1n}\\
     \cdot  &\cdots & \cdot \\
     \cdot  &\cdots & \cdot \\
a_{(n-r)(r+1)} &\cdots & a_{(n-r)n}
\end{pmatrix}= \nonumber\\[0.4cm]
& (-1)^{r(n+1)} \det(A^{u}_{v}). \nonumber
\end{align}
as desired.
\end{proof}

\label{Plucker}

\hyphenation{ge-ne-ra-li-ze}

There are several and more general versions of the Jacobi identity,
(see \cite{grin}, Section 6 and references therein). However, we show now that the  Jacobi identity, as in
(\ref{Jacobi1}), is enough to recover the ordinary
Pl\"ucker relations as in Equation (\ref{ideal-plucker}) of Theorem \ref{fsft}.

\begin{proposition} \label{P1}
Let $X_{i_1\cdots i_r}$ denote the minor of $M\in\mathrm{M}_{r \times p}$ ($r\leq p$) corresponding 
\hyphenation{co-rres-pon-ding} to the submatrix formed by the columns $(i_1,\dots, i_r)$ of $M$. Then, the relations $$
\sum_{k=1}^{r+1} (-1)^{k} X_{i_{1}\cdots i_{r-1}j_{k}}X_{j_{1} \cdots \tilde{\jmath}_{k} \dots j_{r+1}} =0 $$
come directly from the Jacobi identity (\ref{Jacobi1}).
\end{proposition}
\begin{proof}
Fix an $r \times p$ matrix:
\begin{eqnarray*}
    C=\begin{pmatrix}
        a_{11} & a_{12} & \cdots  & a_{1p}\\
        \cdot  & \cdot  & \cdots  & \cdot \\
        \cdot  & \cdot  & \cdots  & \cdot \\
        \cdot  & \cdot  & \cdots  & \cdot \\
        a_{r1} & a_{r2} & \cdots  & a_{rp}
    \end{pmatrix}
\end{eqnarray*}

Set any two sets $\{i_{1},...,i_{r}\}$ and $\{j_{1},...,j_{r}\}$ of ordered indices from $\{1,...,p\}$ and define the following $2r \times 2r$ matrix:
\begin{eqnarray*}
 A= \begin{pmatrix}
        a_{1i_{1}} & \cdots & a_{1i_{r}} & a_{1j_{1}} & \cdots & a_{1j_{r}}\\
        \cdot  & \cdots & \cdot  & \cdot  &\cdots & \cdot \\
        \cdot  & \cdots & \cdot  & \cdot  &\cdots & \cdot \\
         a_{ri_{1}} & \cdot \cdot \cdot  & a_{ri_{r}} & a_{rj_{1}} & \cdot \cdot \cdot  & a_{rj_{r}}\\
         0 & \cdots & 0 & 1 &\cdots & 0\\
         \cdot  & \cdots & \cdot  & \cdot  &\cdots & \cdot \\
         \cdot  & \cdots & \cdot  & \cdot  &\cdots & \cdot \\
         0 & \cdots & 0 & 0 &\cdots & 1\\
    \end{pmatrix} \quad \text{which satisfies} \quad \det A=X_{i_{1}\cdots i_{r}}.
\end{eqnarray*}

\medskip

Suppose that $X_{i_{1}\cdots i_{r}}$ is invertible. Now, fix
$$u=(r+1,\dots, 2r) \qquad \hbox{and} \qquad v=(1,\dots,r),$$
which implies,
$$\Tilde{u}=(1,\dots,r) \qquad \hbox{and} \qquad \Tilde{v}=(r+1,\dots,2r).$$
Therefore,
$$
    A_{v}^{u}=\begin{pmatrix}
         a_{1j_{1}} & \cdots & a_{1j_{r}}\\
         \cdot  &\cdots & \cdot \\
         \cdot  &\cdots & \cdot \\
         a_{rj_{1}} & \cdots & a_{rj_{r}}
    \end{pmatrix} \quad \text{and} \quad \det A_{v}^{u}=X_{j_{1}\cdots j_{r}}.
$$
Putting this in Equation (\ref{Jacobi1}) one gets
\begin{eqnarray}
X_{j_{1}\cdots j_{r}}= (-1)^{r}X_{i_{1}\cdots i_{r}}\det(A^{-1})^{\Tilde{v}}_{\Tilde{u}}. \label{syl1}
\end{eqnarray}
Now,
\begin{eqnarray*}
 (A^{-1})^{\Tilde{v}}_{\Tilde{u}}=\begin{bmatrix}
    (A^{-1})_{kl}
 \end{bmatrix}_{\substack{k=1,...,r \\ l=r+1,...,2r}}
\end{eqnarray*}
Using Cramer's rule, we know that:
\begin{eqnarray*}
(A^{-1})_{kl}=\dfrac{\det A_{k}(e_{l})}{\det(A)}.
\end{eqnarray*}
Therefore, one can easily compute,
\begin{eqnarray*}
(A^{-1})^{\Tilde{v}}_{\Tilde{u}} = \begin{pmatrix}
  -X_{j_{1}i_2\cdots i_{r}}X_{i_{1}\cdots i_{r}}^{-1} & -X_{j_{2}i_{2}\cdots i_{r}}X_{i_{1}\cdots i_{r}}^{-1} &\cdots  & -X_{j_{r}i_{2}\cdots i_{r}}X_{i_{1}\cdots i_{r}}^{-1}\\
   -X_{i_{1}j_{1} i_{3} \cdots i_{r}}X_{i_{1}\cdots i_{r}}^{-1} & -X_{i_{1}j_{2} i_{3}\cdots i_{r}}X_{i_{1}\cdots i_{r}}^{-1} & \cdots  & -X_{i_{1}j_{r}i_{3}\cdots i_{r}}X_{i_{1}\cdots i_{r}}^{-1}\\
   \cdot  & \cdot  & \cdots  & \cdot \\
    \cdot  & \cdot  & \cdots  & \cdot \\
   -X_{i_{1}i_{2}\cdots j_{1}i_{r}}X_{i_{1}\cdots i_{r}}^{-1} & -X_{i_{1}i_{2}\cdots j_{2}i_{r}}X_{i_{1}\cdots i_{r}}^{-1} & \cdots  & -X_{i_{1}i_{2}\cdots j_{r}i_{r}}X_{i_{1}\cdots i_{r}}^{-1}\\
   -X_{i_{1}i_{2}\cdots i_{r-1}j_{1}}X_{i_{1}\cdots i_{r}}^{-1} &  -X_{i_{1}i_{2}\cdots i_{r-1}j_{2}}X_{i_{1}\cdots i_{r}}^{-1} & \cdots  &  -X_{i_{1}i_{2}\cdots i_{r-1}j_{r}}X_{i_{1}\cdots i_{r}}^{-1}
 \end{pmatrix}
\end{eqnarray*}
The Laplace expansion of the determinant, using the first row, gives us:
\begin{eqnarray}
\det(A^{-1})^{\Tilde{v}}_{\Tilde{u}}= (X_{i_{1}\cdots i_{r}})^{-1} \sum_{t=1}^{r} -X_{j_{t}i_{2}\cdots i_{r}} c_{1,t}  \label{Pl}
\end{eqnarray}
where $c_{1,t}$ is the cofactor of the entry $\left((A^{-1})^{\Tilde{v}}_{\Tilde{u}}\right)_{1t}$. Therefore, substituting these in Equation (\ref{syl1}) we get:
\begin{eqnarray}
 X_{j_{1}\cdots j_{r}}= (-1)^{r+1}\sum_{t=1}^{r} X_{j_{t}i_{2}\cdots i_{r}} c_{1,t}. \label{syl2}
\end{eqnarray}
Now, to calculate $c_{1,1}$, consider the same calculations with a matrix $A'$ where we fix $j_{1}=i_{1}$. This gives us:
\begin{eqnarray*}
(A'^{-1})^{\Tilde{v}}_{\Tilde{u}}=\begin{pmatrix}
  -1 & -X_{j_{2}i_{2}\cdots i_{r}}X_{i_{1}\cdots i_{r}}^{-1} & \cdots  & -X_{j_{r}i_{2}..i_{r}}X_{i_{1}\cdots i_{r}}^{-1}\\
  0 & -X_{i_{1}j_{2} i_{3}..i_{r}}X_{i_{1}\cdots i_{r}}^{-1} & \cdots  & -X_{i_{1}j_{r}i_{3}..i_{r}}X_{i_{1}\cdots i_{r}}^{-1}\\
   . & . & \cdots  & .\\
    . & . & \cdots  & .\\
   0& -X_{i_{1}i_{2}\cdots j_{2}i_{r}}X_{i_{1}\cdots i_{r}}^{-1} & \cdots  & -X_{i_{1}\cdots j_{r}i_{r}}X_{i_{1}\cdots i_{r}}^{-1}\\
   0 &  -X_{i_{1}i_{2}\cdots i_{r-1}j_{2}}X_{i_{1}\cdots i_{r}}^{-1} & \cdots  &  -X_{i_{1}i_{2}\cdots i_{r-1}j_{r}}X_{i_{1}\cdots i_{r}}^{-1}
 \end{pmatrix}
\end{eqnarray*}
and from Equation (\ref{Jacobi1}) it follows that:
\begin{eqnarray*}
X_{i_{1}j_{2}\cdots j_{r}}&=& (-1)^{r}X_{i_{1}\cdots i_{r}} \det(A'^{-1})^{\Tilde{v}}_{\Tilde{u}}\\
&=& (-1)^{r+1}X_{i_{1}\cdots i_{r}}\, c_{1,1}
\end{eqnarray*}
Hence, we get:
\begin{eqnarray*}
c_{1,1}=(-1)^{r+1} X_{i_{1}j_{2}\cdots j_{r}} (X_{i_{1}\cdots i_{r}})^{-1}.
\end{eqnarray*}
Similarly, one can get:
\begin{eqnarray*}
c_{1,t}=(-1)^{r+1}X_{j_{1}\cdots j_{t-1}i_{1}j_{t+1}\cdots j_{r}} (X_{i_{1}\cdots i_{r}})^{-1}.
\end{eqnarray*}
By putting this in Equation (\ref{syl2}) we get:
\begin{eqnarray}
 X_{i_{1}\cdots i_{r}} X_{j_{1}\cdots j_{r}}= \sum_{t=1}^{r} X_{j_{t}i_{2}\cdots i_{r}}X_{j_{1}\cdots j_{t-1}i_{1}j_{t+1}\cdots j_{r}}
\end{eqnarray}
which are exactly the Pl\"{u}cker relations in (\ref{ideal-plucker}), up to renaming some indices.
\end{proof}

\medskip
We illustrate this result with an example given below.
\begin{example}
Fix a $2\times 4$ matrix
\begin{eqnarray*}
   C= \begin{pmatrix}
        a_{11} & a_{12} & a_{13} & a_{14}\\
        a_{21} & a_{22} & a_{23} & a_{24}
    \end{pmatrix}
    \end{eqnarray*}
and set $X_{ij}:=\det \begin{pmatrix}
    a_{1i} & a_{1j}\\
    a_{2i} & a_{2j}
\end{pmatrix}.$ Suppose that $X_{12}$ is invertible. Now, consider the following matrix:
\begin{eqnarray*}
   A= \begin{pmatrix}
        a_{11} & a_{12} & a_{13} & a_{14}\\
        a_{21} & a_{22} & a_{23} & a_{24}\\
        0 & 0 & 1 & 0\\
        0 & 0 & 0 & 1
    \end{pmatrix}, \qquad \text{which satisfies} \qquad \det(A)= X_{12}.
\end{eqnarray*}
By fixing $u=(3,4)$ and $v=(1,2)$, we get:
\begin{eqnarray*}
\det A_{v}^{u}=\det \begin{pmatrix}
a_{13} & a_{14}\\
a_{23} & a_{24}
\end{pmatrix}=X_{34}.
\end{eqnarray*}
Moreover, using Cramer's rule, one can easily compute:
\begin{eqnarray*}
(A^{-1})^{\Tilde{v}}_{\Tilde{u}}=\begin{pmatrix}
    (A^{-1})_{13} & (A^{-1})_{14}\\
    (A^{-1})_{23} & (A^{-1})_{24}
\end{pmatrix}=\begin{pmatrix}
    X_{23}X_{12}^{-1} & X_{24}X_{12}^{-1}\\
    -X_{13}X_{12}^{-1} & -X_{14}X_{12}^{-1}
\end{pmatrix}.
\end{eqnarray*}
Hence, substituting these values in Jacobi's identity,
\begin{align*}
&\det(A_{v}^{u})= \det(A) \det(A^{-1})^{\Tilde{v}}_{\Tilde{u}}&
\end{align*}
we get,
\begin{align*}
&X_{34}= X_{12}\det \begin{pmatrix}
X_{23}X_{12}^{-1} & X_{24}X_{12}^{-1} \\
-X_{13}X_{12}^{-1} & -X_{14}X_{12}^{-1}
\end{pmatrix}& \\[0.3cm]
& \Rightarrow \qquad X_{12}X_{34}-X_{13}X_{24}+X_{14}X_{23}=0&
\end{align*}
which is exactly the Pl\"{u}cker relation in this case.
\end{example}

\end{document}